\documentclass[12pt]{amsart}
\usepackage{amssymb,amscd}
\usepackage[all]{xy}
\title[Relative Proportionality]{Relative Proportionality for subvarieties of moduli spaces
of K3 and abelian surfaces} 
\author[S. M\"uller-Stach]{Stefan M\"uller-Stach}
\address{Universit\"at Mainz,
Fachbereich 17, Mathematik,
55099 Mainz, Germany}
\email{mueller-stach@uni-mainz.de}
\email{kzuo@mathematik.uni-mainz.de}
\author[E. Viehweg]{Eckart Viehweg}
\address{Universit\"at Duisburg-Essen, Mathematik, 45117 Essen, Germany}
\email{viehweg@uni-essen.de}
\author[K. Zuo]{Kang Zuo}
\thanks{This work has been supported by the DFG-Leibniz program and by the SFB/TR 45
``Periods, moduli spaces and arithmetic of algebraic varieties''.}
\dedicatory{To Friedrich Hirzebruch}
 \fussy
\begin{document}
\theoremstyle{plain}
\newtheorem{thm}{Theorem}[section]
\newtheorem{theorem}[thm]{Theorem}
\newtheorem{addendum}[thm]{Addendum}
\newtheorem{lemma}[thm]{Lemma}
\newtheorem{lemnot}[thm]{Lemma and Notations}
\newtheorem{corollary}[thm]{Corollary}
\newtheorem{proposition}[thm]{Proposition}
\theoremstyle{definition}
\newtheorem{remark}[thm]{Remark}
\newtheorem{notations}[thm]{Notations}
\newtheorem{assnot}[thm]{Assumptions and Notations}
\newtheorem{remarks}[thm]{Remarks}
\newtheorem{definition}[thm]{Definition}
\newtheorem{claim}[thm]{Claim}
\newtheorem{assumption}[thm]{Assumption}
\newtheorem{assumptions}[thm]{Assumptions}
\newtheorem{properties}[thm]{Properties}
\newtheorem{property}[thm]{Property}
\newtheorem{example}[thm]{Example}
\newtheorem{basic_ex}[thm]{Basic Example}
\newtheorem{conjecture}[thm]{Conjecture}
\newtheorem{constr}[thm]{Construction}
\newtheorem{questions}[thm]{Questions}
\newtheorem{question}[thm]{Question}
\numberwithin{equation}{section}
\newcommand{\sM}{{\mathcal M}}
\newcommand{\sF}{{\mathcal F}}
\newcommand{\sG}{{\mathcal G}}
\newcommand{\sL}{{\mathcal L}}
\newcommand{\bsM}{{\bar{\mathcal M}}}
\newcommand{\sN}{{\mathcal N}}
\newcommand{\sO}{{\mathcal O}}
\newcommand{\C}{{\mathbb C}}
\newcommand{\Q}{{\mathbb Q}}
\newcommand{\N}{{\mathbb N}}
\newcommand{\V}{{\mathbb V}}
\newcommand{\Z}{{\mathbb Z}}
\newcommand{\W}{{\mathbb W}}
\newcommand{\U}{{\mathbb U}}
\newcommand{\SO}{{\rm SO}}
\newcommand{\OO}{{\rm O}}
\newcommand{\SP}{{\rm SP}}
\newcommand{\UU}{{\rm U}}
\newcommand{\SU}{{\rm SU}}
\newcommand{\SL}{{\rm SL}}
\newcommand{\rk}{{\rm rk}}
\newcommand{\ch}{{\rm c}}
\begin{abstract}
The relative proportionality principle of Hirzebruch and H\"ofer was discovered in the case of compactified ball quotient surfaces $X$ when studying 
curves $C \subset X$. It can be expressed as an inequality which attains equality precisely when $C$ is an induced quotient of a subball.
A similar inequality holds for curves on Hilbert modular surfaces. In this paper we prove a generalization of this result to subvarieties of 
Shimura varieties of orthogonal type, i.e. locally symmetric spaces of type ${\mathcal M}=\Gamma \backslash \SO(n,2) \slash K$. 
Furthermore we study the ''inverse problem'' of deciding when an arbitrary subvariety $Z$ of ${\mathcal M}$ 
is of Hodge type, provided it contains sufficiently many divisors $W_i$ which are of Hodge type and satisfy relative proportionality.
\end{abstract}
\maketitle
Let $\sM$ denote a connected Shimura variety of Hodge type associated to a reductive Lie group $G \subset Sp_{2g}$ of Hermitian type defined over $\Q$. A subvariety of $\sM$ is called \textit{special} or a \textit{subvariety of Hodge type}, if it is 
induced by an algebraic subgroup $G_1 \hookrightarrow G$ of Hermitian type.
In particular zero dimensional special subvarieties are just the CM-points. 

As it is well known, a subvariety of Hodge type contains a dense set of CM-points.
The Andr\'e-Oort conjecture states the converse, hence that an irreducible variety $Z$ of $\sM$ is a  subvariety of Hodge type, if the CM points in $Z$ are Zariski dense. Recently Klingler and Yafaev~\cite{ky06} have given a proof of this conjecture, assuming the generalized Riemann Hypothesis. The Andr\'e-Oort conjecture implies immediately also that a subvariety $Z$ of $\sM$ which contains a Zariski dense set of subvarieties of Hodge type must itself be special. We will restrict ourselves to the moduli space $\sM$ of polarized K3 or abelian surfaces, more generally subvarieties of 
Shimura varieties of orthogonal type, i.e. locally symmetric spaces of type ${\mathcal M}=\Gamma \backslash \SO(n,2) \slash K$ with $\Gamma$ a neat arithmetic group. 
For $Z \subset \sM$ we will show that a ``big finite subset'' $\{W_i\}_{i\in I}$ of subvarieties of Hodge type of codimension one is sufficient to force $Z$ to be special. One way to formulate the bigness of the set of subvarieties would be to require the natural map
$$
\pi_1(\bigcup_{i\in I} W_i) \longrightarrow \pi_1(Z)
$$
to be surjective, for a suitable choice of base points. Instead we will consider certain compactifications $\bar{Z}$ of $Z$, and require $\# I$ to be large compared with its Picard number $\rho(\bar{Z})$ and with the
number $\delta(S_{\bar Z})$ of different two by two intersections $S_i\cap S_j$ of irreducible components
$S_i$ and $S_j$ of $S_{\bar Z}$.

A second aspect is the understanding of the ``relative proportionality'', a numerical condition satisfied by subvarieties of Hodge type $W$ of $Z$, provided the universal covering $\tilde{Z}$ of $Z$ is a bounded
symmetric domain. The proportionality principle has been established by Hirzebruch in \cite{hi58} 
for projective manifolds $Z$, and it has been generalized by Mumford in \cite{m77} to the quasi-projective case. To this aim, Mumford used a particularly nice toroidal compactification $\bar{Z}$ of $Z$, constructed in \cite{AMRT75} and extensions of the Hodge bundles to $\bar{Z}$. 
  
In the mid 1980's Hirzebruch and H\"ofer have obtained the relative proportionality inequality for an algebraic curve $\bar{C}$ on an algebraic surface $\bar{Y}$ with universal covering $\tilde{Y}$ a complex ball (see \cite[page 259 and 265]{bhh87}, for example).
A similar inequality holds for curves on Hilbert modular surfaces,
and for special curves the equality was already verified in \cite[\S 4]{hi73}.
\begin{theorem}[Hirzebruch \cite{hi73}; Hirzebruch, H\"ofer \cite{bhh87}]\label{HHP0}
Assume that $S_{\bar Y}=\bar{Y}\setminus Y$ is a strict normal crossing divisor. 
Then for a non-singular curve $\bar{C}\subset Y$ and for the reduced boundary divisor $S_{\bar C}=(\bar{C}\cap S_{\bar Y})_{\rm red}$ one has the relative proportionality inequality saying that
\begin{equation}\label{eqp1}
2\cdot \bar{C}.\bar{C}+2\cdot \deg(S_{\bar C}) \geq -K_{\bar{Y}}.\bar{C}+S_{\bar Y}.\bar{C},
\end{equation}
if $Y$ is a Hilbert modular surface, and
\begin{equation}\label{eqp2}
3\cdot \bar{C}.\bar{C}+3\cdot \deg(S_{\bar C}) \geq -K_{\bar{Y}}.\bar{C}+ 2 \cdot S_{\bar Y}.\bar{C},
\end{equation}
if $Y$ is a ball quotient.\\[.2cm]
If the compactification $\bar{Y}$ is a Mumford compactification, or more generally if
$\Omega^1_{\bar Y}(\log S_{\bar Y})$ is numerically effective (nef) and if
$\omega_{\bar Y}(S_{\bar Y})$ is ample with respect to $Y$, then
the equality in \eqref{eqp1} or in \eqref{eqp2} implies that $\tilde{C}$ is a complex subball of $\tilde{Y}$. 
\end{theorem}

In \cite{bhh87} these inequalities are stated only in the case where $\bar{C}\cap S_{\bar Y}$ intersect transversally. Then they simplify to $2\cdot \bar{C}.\bar{C} \ge -(K_{\bar{Y}}+S_{\bar Y}).\bar{C}$ on Hilbert modular surfaces and $3\cdot \bar{C}.\bar{C} \ge -(K_{\bar{Y}}+S_{\bar Y}).\bar{C}$ on ball quotients.

\medskip
In Section \ref{HHPR} we will prove and generalize those inequalities to certain higher dimensional Shimura varieties $\sM$ which are uniformized by a variation of Hodge structures $\V$ of weight two.
Assuming that the local monodromies at infinity are unipotent, we consider to this aim the Higgs bundle $(E,\theta)$ induced by the Deligne extension of $\V$ to $\bsM$, as explained in
the Notations \ref{del}, and the corresponding Griffiths-Yukawa coupling. For the generalizations of Theorem \ref{HHP0}, stated in Theorem \ref{HHP} we will allow $\sM$ to be a Shimura variety of complex ball type, i.e. $\tilde{\sM}= \SU(n,1)/\UU(n)$, or of type $\SO(n,2)$, i.e. $\tilde{\sM}= \SO(n,2)/\OO(n) \times \UU(1)$, and we replace the curve $C$ by a submanifold $Z$. We will distinguish the different cases corresponding to \eqref{eqp1} and \eqref{eqp2} in Theorem \ref{HHP0}
by posing conditions on the Griffiths-Yukawa coupling. 

For example, a curve $C$ on a Hilbert modular surface $Y$ has 
Griffiths-Yukawa coupling $\theta_{\bar{C}}^{(2)} \neq 0$, whereas for curves in a ball quotient $Y$ it will vanish, since 
already $\theta_{\bar{Y}}^{(2)}=0$.
The remaining case, where the Griffiths-Yukawa coupling satisfies $\theta_{\bar{C}}^{(2)}=0$ but 
$\theta_{\bar{Y}}^{(2)}\neq 0$, only occurs on 
\textit{fake Hilbert modular surfaces}, i.e. on products $X=C_1 \times C_2$ of two curves of genus $g\geq 2$, and for
$C$ a fiber of one of the projections. So to handle this case we should add in Theorem \ref{HHP}:
\begin{equation}\label{eqp3}\hspace{.2cm}
\bar{C}.\bar{C}+\deg(S_{\bar C}) \geq S_{\bar{Y}}.\bar{C},\hspace{.2cm}
\mbox{ if } Y \mbox{ is a product } C_1 \times C_2.
\end{equation}
However the condition $\theta_{\bar C}^{(2)}=0$ on the product of two curves only occurs if
$\bar{C}$ is the fiber of one of the projections, and hence $\bar{C}.\bar{C}=0$ and
$\deg(S_{\bar C}) = S_{\bar{Y}}.\bar{C}$.

\medskip
Let us return to the problem of characterizing subvarieties of Hodge type in $\sM$ by the existence of a big set of special subvarieties. Starting with Section \ref{SSV} we will restrict ourselves to the case where $\sM$ is of type $\SO(n,2)$, and we will consider
$W\subset Z \subset \sM$, with $W$ a Shimura variety and $Z$ unknown. Doing so for Shimura curves $C$ on surfaces $Y\subset \sM$, we get similar expressions as \eqref{eqp1}, \eqref{eqp2} or \eqref{eqp3} with the sign reversed:
\begin{theorem}\label{HHP0R}
Assume that $\bar Y$ is a projective surface, $S_{\bar Y}$ a strict normal crossing divisor and
$Y= \bar{Y}\setminus S_{\bar Y}$, with $\Omega^1_{\bar Y}(\log S_{\bar Y})$ nef and $\omega_{\bar Y}(\log S_{\bar Y})$ ample with respect to $Y$.

Consider a Shimura curve $C\subset Y$ and a closed embedding $Y\subset \sM$. Assume (for simplicity) that 
the second embedding extends to $\bar{Y}\subset \bsM$ for a Mumford compactification $\bsM$ of $\sM$. If $C$ is a Shimura curve, 
then one has:
\begin{gather}\label{eqp4}
\hspace{1.3cm} 2\cdot \bar{C}.\bar{C}+2\cdot\deg(S_{\bar C}) \leq -K_{\bar{Y}}.\bar{C}+S_{\bar Y}.\bar{C} \hspace{1.8cm}
\mbox{ if }\theta_{\bar{C}}^{(2)} \neq 0.\\
\label{eqp5}
\hspace{1.3cm}
3\cdot \bar{C}.\bar{C}+3\cdot \deg(S_{\bar C}) \leq -K_{\bar{Y}}.\bar{C}+ 2 \cdot S_{\bar Y}.\bar{C}
\hspace{1.8cm}
\mbox{ if }\theta_{\bar Y}^{(2)} = 0. \\
\label{eqp6}
\hspace{2.3cm} \bar{C}.\bar{C}+\deg(S_{\bar C}) \leq S_{\bar{Y}}.\bar{C}
\hspace{1.6cm}
\mbox{ if } \theta_{\bar C}^{(2)}=0 \mbox{ but }\theta_{\bar{Y}}^{(2)} \neq 0.
\end{gather}
\end{theorem}
Again the first two inequalities generalize to the higher dimensional case (see Theorem \ref{1.6}). 
Now we can formulate a criterion for $Y$ to be a Shimura surface of Hodge type: 

\begin{theorem}\label{CharShimura}
Consider in Theorem \ref{HHP0R} a finite set of curves $\{C_i\}_{i\in I}$, with     
$$
\# I \geq (\rho({\bar Y})+\delta(S_{\bar Y}))^2 + \rho({\bar Y})+\delta(S_{\bar Y}) + 1,
$$
where $\delta(S_{\bar Y})$ is the number of double points on the boundary and where
$\rho({\bar Y})$ is the Picard number of $\bar{Y}$.
\begin{enumerate}
\item[i)] If for all $i\in I$
$$2\cdot \bar{C}_i.\bar{C}_i+2\cdot\deg(S_{\bar{C}_i}) = -K_{\bar{Y}}.\bar{C}_i+S_{\bar Y}.\bar{C}_i
$$ 
and $\theta_{\bar{C}_i}^{(2)}\neq 0$, then $Y$ is a Hilbert modular surface.
\item[ii)] If for all $i\in I$
$$3\cdot \bar{C}_i.\bar{C}_i+3\cdot \deg(S_{\bar{C}_i}) = -K_{\bar{Y}}.\bar{C}_i+ 2 \cdot S_{\bar Y}.\bar{C}_i
$$
and $\theta_{\bar Y}^{(2)}= 0$, then $Y$ is a ball quotient.
\end{enumerate}
\end{theorem}
In both Theorems, \ref{HHP0R} and \ref{CharShimura}, one can allow the curves $C$ or $C_i$ to be
deformations of Shimura curves in $\sM$, as defined in \ref{def}. 
However such a deformation can only be non-trivial if $\theta_{\bar C}^{(2)}$ or $\theta_{{\bar C}_i}^{(2)}$ are zero. 

As we will see in the proofs, a quasi-projective surface $Y \subset \sM$ containing a
Shimura curve $C$ of Hodge type (or its deformation) and satisfying relative
proportionality  in Theorem \ref{CharShimura} i) or ii), looks in an infinitesimal neighborhood of $C$ like a Shimura surface of the corresponding type.

The corresponding statement in Theorem \ref{3.2} will be formulated for submanifolds $Z\subset \sM$ of arbitrary dimension, but the codimension of the Shimura subvarieties $W\subset Z$, replacing the curves $C$, still has to be one.\vspace{.1cm}

\begin{notations}\label{del}
We consider a projective manifold $\bar Z$, a reduced strict normal crossing divisor
$S_{\bar Z}$ and a variation of Hodge structures $\V$ on $Z=\bar{Z} \setminus S_{\bar Z}$ of weight $k$. 
Even if not stated we will always assume that the local monodromies around the components of $S_{\bar Z}$ are unipotent, and that $\V$ is polarized.
Let $\mathcal V$ be the Deligne extension of $\V\otimes_\C \sO_Z$ to $\bar{Z}$. The $\sF$-filtration
on $Z$ extends to a filtration of $\mathcal V$ by subbundles, and the Gau{\ss}-Manin connection extends to a connection $\nabla$ with logarithmic poles on $\mathcal V$. Griffiths Transversality implies that $\nabla$ induces an $\sO_{\bar Z}$-linear map 
$$
\theta: E:=\bigoplus_{p=0}^k E^{p,k-p}={\mathfrak G}{\mathfrak r}_{\sF}(\mathcal V) \longrightarrow E\otimes \Omega^1_{\bar Z}(\log S_{\bar Z}) =
\big(\bigoplus_{p=0}^k E^{p,k-p}\big)\otimes \Omega^1_{\bar Z}(\log S_{\bar Z}),
$$
with $\theta(E^{p,k-p})\subset E^{p-1,k-p+1}\otimes \Omega^1_{\bar Z}(\log S_{\bar Z})$.
We will call $(E,\theta)$ the \textit{Higgs bundle induced by the Deligne extension of $\V$, and 
$\theta$ the Higgs field}. If we want to indicate the base space, we will write
$(E_{\bar Z},\theta_{\bar Z})$ instead of $(E,\theta)$.

The Higgs field is the direct sum of maps,
$$
\theta_{p,k-p}:E^{p,k-p}\longrightarrow E^{p-1,k-p+1}\otimes \Omega^1_{\bar Z}(\log S_{\bar Z})
$$
Their composite 
$$
\theta^{(k)}_{\bar Z}:=(\theta_{1,k-1} \otimes {\rm id}_{\Omega^1_{\bar Z}(\log S_{\bar Z})}^{\otimes k-1})
\circ \cdots\circ (\theta_{k-1,1}\otimes {\rm id}_{\Omega^1_{\bar Z}(\log S_{\bar Z})})\circ
\theta_{k,0},
$$
called the \textit{Griffiths-Yukawa coupling}, has image in 
$E^{0,k}\otimes S^k(\Omega^1_{\bar Z}(\log S_{\bar Z}))$.\vspace{.2cm} 

Let $\bsM$ be a second projective manifold and let $\sM$ be the complement of a reduced strict normal crossing divisor $S_\bsM$. We will consider a morphism $\varphi: Z \to \sM$, generically finite over its image. We will denote the induced rational map ${\bar Z} \to \bsM$ again by $\varphi$.

The rational map $\varphi:\bar{Z} \to \bsM$ is given by a morphism $\varphi_0$ on the complement
$\bar{Z}_0$ of a codimension two subscheme. For a locally free sheaf $\sF$ on $\bsM$ we will write 
$\varphi^* \sF$ for the maximal extension of the pullback $\varphi_0^* \sF$ from $\bar{Z}_0$ to $\bar{Z}$.
Correspondingly, If $B$ is a reduced divisor, $\varphi^* B$ will be the closure of $\varphi_0^* B$.

The inclusion $T_{\bar{Z}_0}(-\log S_{{\bar Z}_0})\to \varphi^*T_{\bsM}(-\log S_\bsM)|_{\bar{Z}_0}$ 
extends to $\bar Z$ and we define the logarithmic normal sheaf $\check{N}_{\bar{Z}/\bsM}$ by the exact sequence
\begin{equation}\label{taut}
0 \longrightarrow T_{\bar{Z}}(-\log S_{{\bar Z}})\longrightarrow\varphi^* T_{\bsM}(-\log S_\bsM)\longrightarrow \check{N}_{\bar{Z}/\bsM}\longrightarrow 0.
\end{equation}
Even if $\bar{Z}$ is a submanifold of $\bsM$ the logarithmic normal sheaf might differ from the usual
normal sheaf $\sN_{{\bar Z}/\bsM}$ defined as the cokernel of $T_{\bar{Z}} \to \varphi^* T_{\bsM}$.
Comparing the Chern classes one obtains for $\bar{Z}\subset \bsM$
\begin{equation}\label{tautchern}
\ch_1(\check{N}_{\bar{Z}/\bsM})= 
\ch_1(\sN_{{\bar Z}/\bsM}) - S_{\bsM}|_{\bar Z} + S_{\bar Z}.
\end{equation}
More generally, assume that $\varphi:Z \to \sM$ is \'etale over its image of degree
$\deg(\varphi)$ and that $\dim(\bsM)=n=\dim(\bar Z)+1$. Writing 
$\varphi(\bar{Z})$ for the closure of the image of $\bar{Z}$ in $\bsM$, one finds for all divisors $L$ on $\bsM$ 
\begin{equation}\label{tautchern2}
\ch_1(\check{N}_{\bar{Z}/\bsM}).(\varphi^*L)^{n-2} =
\deg(\varphi)\cdot \varphi(\bar{Z})^2.L^{n-2} - (S_{\bsM}|_{\varphi(\bar Z)} - (S_{\bsM}|_{\varphi(\bar Z)})_{\rm red}).L^{n-2}.
\end{equation}
In fact, both sides are compatible with blowing ups of $\bar{Z}$ with centers in $S_{\bar Z}$.
So we may assume that $\varphi$ is a morphism. The formula \eqref{tautchern2} holds, if
$\varphi$ is injective, and the general case follows from the projection formula.\vspace{.1cm}
Usually $\sM$ and $W$ will denote Shimura varieties, and $Z$ will map to in $\sM$, or $W$ will map to $Z$.
\vspace{.2cm}\end{notations}
{\bf Acknowledgements.}\\[.1cm]
We would like to thank Sheng-Li Tan and De-Qi Zhang for discussions on Lemma 4.1.

\section{Shimura varieties of type $\SO(m,2)$ and $\SU(m,1)$}\label{type}

Let us first recall some well known basic facts on connected Shimura varieties and their connected Shimura subvarieties (see also \cite{An01} or \cite{Mil04}). We consider
\begin{gather*}
H \mbox{ a connected semisimple group defined over } \mathbb{Q} \mbox{ and of Hermitian type,}\\
K\subset H(\mathbb R) \mbox{ a  maximal compact subgroup} \\
X^+=H(\mathbb R)/K \mbox{ a  bounded  symmetric  domain for } H.
\end{gather*}
Writing $H^+(\mathbb R)$ for the connected component of $1$ in $H(\mathbb R)$, one can consider $X^+$ 
as a conjugacy class of 1-parameter subgroups $\UU(1)\to H^+(\mathbb R)$.

Choose any $\mathbb Z$-structure $H_{\mathbb Z}$ on $H$ and let $\Gamma\subset H(\Q)$ be an arithmetic subgroup, i.e. a subgroup $\Gamma$ which is commensurable to $H_{\mathbb Z}(\mathbb Z)$. 
In addition we will always assume that $\Gamma$ is neat. By a theorem of Baily-Borel 
the analytic space $\mathcal M^{\rm an}:=\Gamma\setminus X^+$ admits the so-called Baily-Borel compactification 
$\bar {\mathcal M}^*=\mathcal M\cup \Delta^*$ by adding the cusps $\Delta^*$ at infinity in $\mathcal M$.
Since $\sM^*$ is projective, $\mathcal M^{\rm an}$ has the structure of an algebraic variety $\sM$ over $\C$, inducing the analytic space structure on $\mathcal M^{\rm an}$. Since $\Gamma$ is torsion-free
$\mathcal M$ is smooth.
\vspace{.1cm}

We will call $\mathcal M$ a connected Shimura variety, although one sometimes requires in addition that 
$\Gamma\subset H_{\mathbb Z}(\mathbb Z)$ is a congruence subgroup, i.e. that
$\Gamma$ contains the kernel of $H_{\mathbb Z}(\mathbb Z)\to H_{\mathbb Z}(\mathbb Z/N \mathbb Z)$ for some $N$. 

Now let $G$ be a connected reductive group over $\Q$, and such that
$H=G^{ad}=G/Z(G)$ is of Hermitian type. Then $X^+$ is a $G^+(\mathbb R)$-conjugacy class of 1-parameter subgroups.\vspace{.1cm}

If we have another group $G_1$ of Hermitian type, and a homomorphism $G_1\to G$ sending conjugacy classes $X_1^+\to X^+$, then 
the map $X_1^+\to X^+$ is holomorphic and totally geodesic by Satake.\vspace{.1cm} 

A Shimura subvariety of Hodge type (also called special subvariety) is a component of the image of some $X_1^+$ in $\mathcal M$.  

The following theorem gives a characterization of the quasi-projective subvarieties of $\mathcal M$ which are Shimura subvarieties of Hodge type or deformations of those. There is also the more general notion of Kuga fiber spaces (used in \cite{mvz07}) and of their bases. 
These are subvarieties of Hodge type if they contain a point corresponding to a CM abelian variety.

\begin{theorem}[Abdulali \cite{Abd94}, Moonen \cite{Mo98}, see also \cite{mvz07}, Section 1]\label{AbdMo} \ \\ 
Let $W \subset \sM$ be a closed algebraic totally geodesic embedding. Then $\sM$ contains a Shimura subvariety of Hodge type, isomorphic to $W\times W'$ (up to a finite \'etale cover). 
In particular, if $W$ is rigid, hence if $W'$ is a point, $W$ is a Shimura subvariety of Hodge type.
\end{theorem}
\begin{notations}\label{def}
We will consider Shimura varieties \textit{up to \'etale coverings} so we will allow to replace $\Gamma$ by a 
subgroup of finite index, whenever necessary. By abuse of notations we will call $W$ a Shimura subvariety of Hodge type if $\sigma(W)$ has this property and if $\sigma:W\to \sigma(W)$ is \'etale.

A subvariety $\iota:W \hookrightarrow \sM$ will be called {\em a deformation of a Shimura subvariety of Hodge type,} if there exists a connected scheme $W'$, points $w'_1$, $w'_2 \in W'$ and a morphism $\Psi:W\times W' \to \sM$, such that $\iota=\Psi|_{W\times \{w'_1\}}$ and such that $\Psi|_{W\times \{w'_2\}}$ is an embedding
whose image is a Shimura  subvariety of Hodge type.
\end{notations}

Next, following \cite{Kud03}, we introduce Shimura varieties of orthogonal type determined by the following data:
\begin{gather*}
V, ( \ , \ ) \mbox{ an inner product space over } \Q \mbox{ of signature }  (n,2),\\
G=\SO(n,2)=\big\{g\in \SL_{n+2}|(g(x),g(y))=(x,y),\,\forall x, y\in V\big\}.
\end{gather*}
In number theory one often prefers to work with the isogenous group ${\rm GSpin}(V)$. One defines the $n$-dimensional complex space
$$
D=\big\{w\in V(\C)|(w,w)=0, (w,\bar w)<0\big\}/\C^*\subset \mathbb P(V(\C)),
$$
which is the union $D=D^+\cup D^-$ of two copies of the bounded symmetric domain 
$\SO(n,2)/\OO(n)\times \UU(1)$ of type IV, interchanged by complex conjugation.

Fixing a $\Z$--structure $G_\Z(\Z)$ on $G$ and again a neat arithmetic subgroup $\Gamma\subset G(\mathbb Q)$, the quotient $\mathcal M:=\Gamma\setminus D^+$ is called a connected Shimura variety of orthogonal type. 

The inner product space $V, ( \ , \ )$ together with the $\mathbb Z$-structure descends to a
polarized variation of Hodge structure $\V$ with a $\mathbb Z$-structure of weight two over $\mathcal M$. The one-dimensional vector spaces  
$V_{w}$ lying over $[w] \in D^+$ define the Hodge bundle $E^{2,0}$ and their complex conjugates 
${\bar V}_{w}$ define $E^{0,2}$. The orthogonal complements of the span $<V_w,{\bar V}_{w}>$, for $w\in D^+$ define $E^{1,1}$. 
It is also known that the Kodaira-Spencer (or dual Higgs field) $\theta: E^{2,0}\otimes T_{\mathcal M}\longrightarrow E^{1,1}$ is an isomorphism. 

Assuming that the local monodromies around the cusps are unipotent, and that $\Gamma$ is neat, Mumford studied in \cite{m77} smooth toroidal compactifications $\bar{\mathcal M}$ with $S_\bsM=\bsM\setminus\sM$ a normal crossing divisor, constructed in \cite{AMRT75}. The Higgs bundle $(E,\theta)$ extends to a unique  logarithmic Higgs bundle on $\bsM$, denoted again by $(E,\theta)$. In fact, as discussed in \cite[Section 2]{mvz07}, the bundle $(E,\theta)$ coincides with the one induced by the Deligne extension 
of $\V$ and the induced dual Higgs field
$$
E^{2,0}\otimes T_{\bar{\mathcal M}}(-\log S_\bsM)\longrightarrow E^{1,1}
$$
is still an isomorphism.\vspace{.1cm}

To define Shimura subvarieties of $\mathcal M$, as in \cite{Kud03}, one starts with a set of $\Q$-linearly independent vectors $x=\{x_1,\ldots , x_r\}\subset V(\Q)$
such that the intersection matrix $\big((x_i,x_j)\big)_{i,j}$ is positive definite.  We define
$V_x$ to be the orthogonal complement of the span $<x_1,\ldots,x_r>$, and $G_x$ to be the stabilizer of the span $<x_1,\ldots,x_r>$. The operation on $V_x$ defines an isomorphism $G_x\cong \SO(n-r,2)$.

The embedding $\SO(n-r,2)\hookrightarrow \SO(n,2)$ of groups 
induces the totally geodesic holomorphic embedding of the corresponding bounded symmetric domains 
$$
\xymatrix{
\SO(n-r,2)/\OO(n-r)\times \UU(1)\ar@{^{(}->}[r] &  \SO(n,2)/\OO(n)\times \UU(1).}
$$
The image of $\SO(n-r,2)/\OO(n-r)\times \UU(1)$ in $\mathcal M=\Gamma\setminus \SO(n,2)/\OO(n)\times \UU(1)$ 
is a Shimura subvariety $W$ of orthogonal type (see \cite{Kud03}, Page 4, (2.6) and (2.8)).\vspace{.1cm}

The pullback of the variation of Hodge structures $\V$ to $W$ decomposes as ${\mathbb W} \oplus\mathbb U$,
where $\mathbb U$ corresponds to a unitary variation of Hodge structures of bidegree $(1,1)$ with the $\mathbb Z$-structure arising from the $\Q$-subspace $<x_1,\ldots,x_r>$. Hence, after taking a finite \'etale base change we may assume $\mathbb U$ is trivial. Correspondingly, one obtains a decomposition of the Higgs bundle of the variation of Hodge structures
$$
(E^{2,0}_W\oplus E^{1,1}_W\oplus E^{0,2}_W,\theta)\oplus (\mathcal O_W^{\oplus r},0),
$$
where $E^{2,0}_W$ and $E^{0,2}_W$ are the restrictions of the invertible sheaves
$E^{2,0}$ and $E^{0,2}$ to $W$. Choosing as above a smooth Mumford compactification $\bar W$ of $W$
with $S_{\bar W}=\bar W\setminus W$ a strict normal crossing divisor, the logarithmic Higgs field defines an isomorphism 
$$
\theta: E^{2,0}_{\bar W} \otimes  T_{\bar W}(-\log S_{\bar W})\longrightarrow  E^{1,1}_{\bar W}.
$$
The Shimura subvariety $W$ is rigid in $\mathcal M$, since the Griffiths-Yukawa coupling does not vanish \cite{mvz07}. 
Hence by Theorem \ref{AbdMo} it is of Hodge type and of type $\SO(n-r,2)$, or as we will sometimes say,
of Hodge type for $\SO(n-r,2)$.\vspace{.1cm}

A Shimura variety is of Hodge type for $\SU(n,1)$, if the associated Hermitian symmetric space is the $n$-dimensional complex ball
$$
X^+=\SU(n,1)/\UU(n).
$$
\begin{remark}\label{weight2}
In this case the natural uniformizing variation of Hodge structures is of weight one and the Higgs bundle has the form $(H^{1,0}\oplus H^{0,1},\tau)$,
where $H^{1,0}$ is a line bundle and where $\tau: T_{\mathcal M}\otimes H^{1,0}\to H^{0,1}$ is an isomorphism.

However if $\Gamma\backslash \SU(n,1) \slash \UU(n)$ occurs as a Shimura subvariety of Hodge type in some Shimura variety $\sM$ of type $\SO(n,2)$, then the restriction $(E,\theta)$ of the uniformizing variation of Hodge structures on $\sM$ will be of weight two. The corresponding Higgs bundles are related by 
$$
E^{2,0}=H^{1,0}, \ \ E^{1,1}=H^{0,1}\oplus {H^{0,1}}^\vee \mbox{ \ and \ }E^{0,2}={H^{1,0}}^\vee.
$$ 
\end{remark}
\begin{example} Kondo \cite{kondo} has constructed a moduli embedding of a compact Shimura surface of type $\SU(2,1)$ (appearing in Deligne-Mostow's list as a component in the moduli space parameterizing Jacobian of genus $6$ admitting CM of $\Q(e^{\frac{2\pi i}{5}})$) into a Shimura variety of type $\SO(10,2)$  parameterizing a subfamily of K3 surfaces.
\end{example}

\begin{lemnot}\label{decomposition} 
Let $\sigma: W \to \mathcal M=\Gamma\backslash\SO(n,2) \slash \OO(n)\times \UU(1)$ be a generically finite morphism from a non-singular $m$-fold $W$ and let $\bar{W}$ be a projective compactification of $W$ with $S_{\bar W}=\bar{W}\setminus W$ a strict normal crossing divisor and with $\omega_{\bar{W}}(S_{\bar W})$ nef and big. Assume that the local monodromies around the components of $S_{\bar W}$ are unipotent, and  write $\sigma^*{\V}=\W\oplus \U$ where $\mathbb U$ is the largest unitary subvariation of Hodge structures of type $(1,1)$.

For the logarithmic Higgs bundle $(E_{\bar W},\theta_{\bar W})$ induced by the Deligne extension of $\W$ to $\bar{W}$, let
$E^{1,1}_\diamond$ denote the image of the Higgs map
$$
\theta: T_{\bar W}(-\log S_{\bar W})\otimes E^{2,0}_{\bar W}\longrightarrow E^{1,1}_{\bar W}.
$$
\begin{enumerate}
\item[i)] The following conditions are equivalent:
\begin{enumerate}
\item[a)] $E^{1,1}_{\diamond}=E^{1,1}_{\bar W}$.
\item[a')] For the Higgs bundle $(E_{W},\theta_{W})=(E_{\bar W},\theta_{\bar W})|_W$ one has
$E^{1,1}_{\diamond}|_W=E^{1,1}_{W}$.
\item[b)] $W$ is a Shimura subvariety of Hodge type for $\SO(m,2)$.
\item[c)] $W$ is a Shimura subvariety of Hodge type and the Griffiths-Yukawa coupling $\theta_{\bar W}^{(2)}$
of $(E_{\bar W},\theta_{\bar W})$ is non-zero.
\end{enumerate}
Moreover, if the conditions a), b) and c) hold true,
\begin{equation*}
\ch_1(\omega_{\bar W}(S_{\bar W}))= m \cdot \ch_1(E^{2,0}_{\bar W}).
\end{equation*}
\item[ii)]  The following conditions are equivalent:
\begin{enumerate}
\item[a)] The Higgs bundle $(E_{\bar W},\theta_{\bar W})$ decomposes as a direct sum 
$$
(E_{\bar W},\theta_{\bar W})= (E^{2,0}_{\bar W} \oplus E^{1,1}_{\diamond}, \theta_{\diamond}) \oplus (E'^{1,1}_{\diamond} \oplus E^{0,2}_{\bar W},\theta'_{\diamond}),
$$
where $E^{1,1}_{\bar W}=E^{1,1}_\diamond \oplus E'^{1,1}_{\diamond}$ and $E'^{1,1}_{\diamond}\neq 0$. 
\item[a')] The Higgs bundle $(E_{W},\theta_{W})=(E_{\bar W},\theta_{\bar W})|_W$ decomposes as a direct sum 
$$(E_{W},\theta_{W})= (E^{2,0}_{W} \oplus E^{1,1}_{\diamond}|_W, \theta_{\diamond}) \oplus (E'^{1,1}_{\diamond}|_W \oplus E^{0,2}_{W},\theta'_{\diamond}).$$ 
\item[b)] $W$ is a deformation of a Shimura subvariety of Hodge type for $\SU(m,1)$.
\item[c)] $W$ is a deformation of a Shimura subvariety of Hodge type and the Griffiths-Yukawa coupling $\theta_{\bar W}^{(2)}$
of $(E_{\bar W},\theta_{\bar W})$ is zero.
\end{enumerate}
Moreover, if the conditions a), b) and c) hold true,
\begin{equation*}
\ch_1(\omega_{\bar W}(S_{\bar W}))= (m+1) \cdot \ch_1(E^{2,0}_{\bar W}),
\end{equation*}
and if $\dim(W)>1$ then $W$ is rigid, hence of Hodge type. 
\end{enumerate}
\end{lemnot}
\begin{proof} Assume first, that $\sigma(W)\subset \sM$ is a deformation of a Shimura subvariety of Hodge type. Then one has the isomorphism in i), a') or the decomposition in ii), a'). We will show, that
this extends to $\bar{W}$, in particular this will imply that in i) or ii) the conditions a) and a') are equivalent.\vspace{.2cm} 

Let ${\bar W}'$ be a Mumford compactification and $S_{{\bar W}'}={\bar W}'\setminus W'$. 
Then as discussed in \cite[Section 2]{mvz07} the image of the Higgs field 
$E^{1,1}_{\diamond,{\bar W}'}$ is a direct factor of $E^{1,1}_{{\bar W}'}$. Choose a third compactification
$\hat{W}$ of $W$ which allows morphisms $\phi:\hat{W} \to {\bar W}$ and $\phi':\hat{W} \to {\bar W}'$.
Since the Deligne extension is compatible with pullbacks, one has
$$
\phi^*T_{\bar W}(-\log S_{\bar W}) \hookrightarrow \phi^*E^{1,1}_{\bar W}\otimes {E^{2,0}_{\bar W}}^{-1}=
\phi'^*E^{1,1}_{{\bar W}'}\otimes {E^{2,0}_{{\bar W}'}}^{-1} \hookleftarrow \phi'^*T_{{\bar W}'}(-\log S_{{\bar W}'}).
$$
The inclusion on the right hand side splits, hence we obtain an inclusion
$$
\phi^*T_{{\bar W}}(-\log S_{{\bar W}}) \hookrightarrow \phi'^*T_{{\bar W}'}(-\log S_{{\bar W}'}).
$$
This must be an isomorphism, since as in the proof of \cite[Lemma 2.7]{mvz07} it is easy to see that
$$
\phi^*\omega_{\bar W}(S_{\bar W})=\phi'^*\omega_{{\bar W}'}(S_{{\bar W}'}).
$$
So $\phi^*T_{\bar W}(-\log S_{\bar W})$ is a direct factor of $\phi^*E^{1,1}_{{\bar W}}\otimes {E^{2,0}_{{\bar W}}}^{-1}$, and hence 
$$
T_{\bar W}(-\log S_{\bar W}) \hookrightarrow E^{1,1}_{{\bar W}}\otimes {E^{2,0}_{{\bar W}}}^{-1}
$$ 
splits, as claimed in i), a) and ii), a).\vspace{.2cm} 

Assuming the condition a) in i) or ii), we will write  
\begin{equation}\label{eqdecomp}
(E_{\bar W},\theta_{\bar W})= (E_\diamond,\theta_\diamond)\oplus (E'_\diamond,\theta'_\diamond), 
\end{equation}
where the first direct factor contains $E^{2,0}$, hence $T_{\bar W}(\log S_{\bar W})\otimes E^{2,0}$ as well.

We will show next, that the existence of this splitting of Higgs bundles forces
$W$ to be the deformation of a Shimura subvariety of Hodge type. The decompositions, corresponding to $\sigma^*{\V}=\W\oplus \U$ or to the one in part ii), are both orthogonal with respect to the Hodge metric on the universal covering $\mathcal M$. The restriction of the Higgs map
$$ 
\theta: T_{\bar W}(-\log S_{\bar W}) \longrightarrow E^{1,1}_{\diamond} \otimes E^{2,0\vee}_{\bar W}
$$
to $W$ is then an isomorphism, and it can be identified with the differential 
$$
d\sigma: T_W\longrightarrow d\sigma(T_W)\subset \sigma^*T_{\mathcal M}\simeq E^{1,1}_{\bar W}\otimes E^{2,0\vee}_{\bar W}\oplus U^{1,1} \otimes E^{2,0\vee}_{\bar W},
$$
where $U^{1,1}$ is the Higgs bundle associated to $\U$.
Hence the image of $d\sigma$ is a holomorphic direct factor, and orthogonal with respect to the Hodge metric.
Therefore $\sigma$ is \'etale over its image and the latter is a non-singular subvariety of $\sM$.
Since $\sigma(W) \subset \mathcal M$ is a complete submanifold with respect to the Hodge metric $h$ (i.e. every Cauchy sequence in the sub metric space $(W,h_W)\subset (\mathcal M, h)$ converges to a point in $W$), then \cite[Claim 6.9]{mvz07} together with \cite[Theorem I.14.5]{He62} show that $\sigma(W) \hookrightarrow \sM$ is a holomorphic totally geodesic  embedding. Then $W$ is either uniformized by $\SO(m,2)/\OO(m)\times \UU(1)$ or by $\SU(m,1)/\UU(m)$, where $\SO(m,2)$ respectively $\SU(m,1)$ is the non-compact factor of the Zariski closure of the monodromy group. By Remark \ref{weight2} the Griffiths-Yukawa coupling is zero for $W$ of type $\SU(m,1)$ and non-zero for $W$ of type $\SO(m,2)$.

The non-vanishing of the Griffiths-Yukawa coupling implies the rigidity of $W$, and by
Theorem \ref{AbdMo} $W$ is a Shimura variety of Hodge type in this case.

By \cite{sz} the same holds true if the Griffiths-Yukawa coupling vanishes, and if $\dim(W)>1$.

Returning to the splitting in \eqref{eqdecomp} one has 
$E^{0,2}_\diamond=E^{0,2}$ if and only if the Griffiths-Yukawa coupling does not vanish.
By the choice of $\W$ this is equivalent to $E'_\diamond=0$. So we verified the equivalence of the conditions a), b) and c) in i) and in ii).

It remains to verify the description of $\ch_1(\omega_{\bar W}(S_{\bar W}))$. In i)
we have seen already that $E^{0,2}_\diamond=E^{0,2}$, hence $E'_\diamond$ is concentrated in bidegree $(1,1)$ and $\theta'_\diamond=0$. By the choice of $\W$ this is only possible if
$E'_\diamond=0$. Then  
\begin{gather*}
\ch_1(E^{2,0}_{\bar W}\oplus E^{1,1}_\diamond\oplus E^{0,2}_{\bar W})=\ch_1( E^{1,1}_\diamond)=0\\
\mbox{and \ \ } 
\ch_1(T_{\bar W}(-\log S_{\bar W}))+m \cdot \ch_1(E^{2,0}_{\bar W}) =0.
\end{gather*}
In Case ii) the Griffiths-Yukawa coupling is zero. So the 
Higgs subbundle $(E_\diamond,\theta_\diamond)$ is concentrated in bidegrees $(2,0)$ and $(1,1)$.
Since $(E_{\bar W},\theta_{\bar W})$ is self dual, one finds that $E'^{1,1}_{\diamond}={E^{1,1}_{\diamond}}^\vee$. For the Chern classes this implies that 
\begin{gather*}
\ch_1(E^{2,0}_{\bar W}\oplus E^{1,1}_\diamond)=0 \mbox{ \ \ and hence}\\
\ch_1(T_{\bar W}(-\log S_{\bar W}))+ (m+1) \cdot \ch_1(E^{2,0}_{\bar W}) =0.
\end{gather*}\nopagebreak\end{proof}
\begin{remarks} \
\begin{enumerate}
\item[1.] The Shimura subvarieties in Lemma \ref{decomposition} include all rigid Shimura subvarieties of Shimura varieties of orthogonal type. 
\item[2.] For $n=19$ $\sM$ is the moduli scheme of polarized K3 surfaces \cite{ks67}. 
The Kummer construction identifies ${\mathcal A}_2$ with a Shimura subvariety of $\SO(3,2)$ type. 
For $n=1$ and $2$ one recovers modular curves, Hilbert modular surfaces and their quaternionic versions \cite{Kud03}.
\item[3.] If a Satake embedding $\mathcal M\to \mathcal A_g$ into a Shimura variety of $Sp(2g,\mathbb R)$-type (i.e. into the moduli space of polarized abelian varieties with a suitable level structure) is of Hodge type, then it maps Shimura subvarieties of Hodge type to Shimura subvarieties of Hodge type \cite{Abd94}.
The Kuga-Satake construction, see \cite{ks67} and \cite{vG00}, provides us with such an embedding.
\end{enumerate}
\end{remarks}

\section{Hirzebruch-H\"ofer's relative proportionality on Shimura varieties of type $\SU(n,1)$ or $\SO(n,2)$.}\label{HHPR}

In this section we will study subvarieties $Z$ of a Shimura variety $\sM$ of type $\SO(n,2)$ or
$\SU(n,1)$. We want to understand numerical conditions on natural sheaves on certain compactifications, generalizing the relative Hirzebruch-H\"ofer Proportionality stated in Theorem \ref{HHP0}.
\begin{assnot}\label{assSect2}
Let $\sM$ be a Shimura variety of type $\SO(n,2)$ or $\SU(n,1)$, and let $\bsM$ be a smooth Mumford compactification of $\sM$ with $S_\bsM=\bsM\setminus\sM$ a strict normal crossing divisor. We denote by $\V$ the uniformizing weight two variation of Hodge structures on $\sM$, and we will assume that the local monodromies around the components of $S_\bsM$ and $S_{\bar Z}$ are unipotent. We write again $\V=\W\oplus \U$ where $\U$ is the maximal
unitary subvariation of Hodge structures.

As in the Notations \ref{del} let $\bar Z$ be a smooth projective $d$-dimensional variety, $S_{\bar Z}$ a reduced strict normal crossing divisor on $\bar Z$ and write $Z={\bar Z}\setminus S_{\bar Z}$.
We consider a morphism $\varphi: Z \to \sM$ generically finite over its image and the induced rational map
${\bar Z} \to \bsM$ again denoted by $\varphi$.
 
The Higgs bundle of the weight two variation of Hodge structures $\W$ on $\sM$ will be denoted
by $(E_\bsM,\theta_\bsM)$, whereas the Higgs bundle of the pullback of $\W$ to $Z$ is written as $(E_Z,\theta_Z)$. Let $(E_{\bar{Z}},\theta_{\bar{Z}})$ be the Higgs bundle induced by the Deligne extension
of $\varphi^*\W$, so $(E_Z,\theta_Z)=(E_{\bar{Z}},\theta_{\bar{Z}})|_Z$.

We will assume that $\Omega^1_{\bar Z}(\log S_{\bar Z})$ nef and that $\omega_{\bar{Z}}(S_{\bar Z})$ is ample with respect to $Z$.
\end{assnot}

The assumption, that $\Omega^1_{\bar Z}(\log S_{\bar Z})$ is nef and that $\omega_{\bar{Z}}(S_{\bar Z})$ is ample with respect to $Z$, will allow to apply Yau's Uniformization Theorem (\cite{ya93}, see also \cite[Section 1]{vz05}) to $Z$. As discussed in
\cite[Lemma 4.1]{vz05} and \cite[\S 2]{mvz07} this assumption automatically holds true for compact submanifolds of Shimura varieties, and it holds if $Z$ is a Shimura subvariety of Hodge type, and $\bar{Z}$ a Mumford compactification.

In general, for a rational map between manifolds, the pullback of the logarithmic tangent sheaf will not be locally free. However, since in our situation $\bsM$ is a Mumford compactification of a Shimura variety this will be the case. 
\begin{lemma}\label{pullback} We keep the assumptions made in \ref{assSect2}. 
\begin{enumerate}
\item[a.] The sheaf $\varphi^* T_{\bsM}(-\log S_\bsM)$ is locally free and isomorphic to a direct factor $E^{1,1}_{\diamond \bar{Z}}\otimes {E^{2,0}_{\bar Z}}^{-1}$ of 
$E^{1,1}_{\bar Z}\otimes {E^{2,0}_{\bar Z}}^{-1}=
E^{1,1}_{\bar Z}\otimes {E^{0,2}_{\bar Z}}$.\vspace{.1cm}
\item[b.] If $\sM$ is of type $\SO(n,2)$, then $E^{1,1}_{\bar Z}=E^{1,1}_{\diamond \bar{Z}}$.\vspace{.1cm}
\item[c.] If $\sM$ is of type $\SU(n,1)$, then $E^{1,1}_{\bar Z}=E^{1,1}_{\diamond \bar Z}\oplus {E^{1,1}_{\diamond \bar Z}}^\vee$.\vspace{.1cm}
\item[d.] If $Z$ is the deformation of a Shimura subvariety of Hodge type, 
then $\check{N}_{\bar{Z}/\bsM}$ is locally free and 
$$
\varphi^* T_{\bsM}(-\log S_\bsM)\cong T_{\bar{Z}}(-\log S_{{\bar Z}})\oplus \check{N}_{\bar{Z}/\bsM}.
$$
\end{enumerate} 
\end{lemma}
\begin{proof}
If $\varphi$ is an isomorphism, hence if $\bar{Z}=\bsM$, the properties a), b) and c) have been verified in Lemma \ref{decomposition}. 

Let $\bar{Z}_0$ denote the largest open subscheme of $\bar Z$ for which  
$\varphi^{-1}(S_\bsM)|_{{\bar Z}_0}$ is a non-singular divisor and $\varphi|_{{\bar Z}_0}$ a morphism.
The Deligne extension is compatible with pullback under morphisms, and a), b) and c) hold true on $\bar{Z}_0$. Knowing this, and using the fact that the Higgs bundles induced by the Deligne extension of a variation of Hodge structures are locally free, one obtains a) and the description of the Higgs bundles in b) and c) extend to $\bar{Z}$.
  
The decomposition in Part d) follows, since in this case both,
$$
T_{\bar Z}(-\log S_{\bar Z})\mbox{ \ \ and \ \ }\varphi^* T_{\bsM}(-\log S_\bsM),
$$ 
are direct factors of $E^{1,1}_{\bar Z}\otimes {E^{2,0}_{\bar Z}}^{-1}.$
\end{proof}
We will need the Simpson correspondence, hence the notion of slopes of coherent sheaves.
Let $\sL$ be an invertible sheaf, nef and ample with respect to $Z$.
For any rank $r$ coherent sheaf $\sF$ on $\bar Z$ define the degree and the slope with respect to $\sL$ as
\begin{equation}\label{slope}
\deg_{\sL}(\sF) := \ch_1(\sF)\cdot \ch_1(\sL)^{d-1} \mbox{ \ \ and \ \ } 
\mu_{\sL}:=\frac{\deg_{\sL}(\sF)}{r}.
\end{equation}
As we will see in the next Theorem, the generalized Hirzebruch-H\"ofer inequality is an inequality of Arakelov type similar to those considered in \cite{stz03} and \cite{vz03} over curves and in \cite{vz05} and \cite{mvz07} for variations of Hodge structures of weight one.  
\begin{theorem} (Hirzebruch-H\"ofer's relative proportionality inequality)\label{HHP} \\
Keeping the assumptions and notations stated in \ref{assSect2}, one finds:
\begin{enumerate}
\item[ i)] If $\sM$ is of $\SO(n,2)$-type and if the Griffiths-Yukawa coupling $\theta_{\bar Z}^2\neq 0$ then
\begin{multline*}\hspace*{1cm}
d \cdot \deg_{\omega_{\bar Z}(S_{\bar Z})}(\check{N}_{\bar Z/\bsM})+ (n-d)\cdot \deg_{\omega_{\bar Z}(S_{\bar Z})}(\Omega^1_{\bar Z}(\log S_{\bar Z})) =\\ n\cdot\big( \deg_{\omega_{\bar Z}(S_{\bar Z})}(\Omega^1_{\bar Z}(\log S_{\bar Z})) - d\cdot \deg_{\omega_{\bar Z}(S_{\bar Z})}(E^{2,0}_{\bar Z})\big)\geq 0.
\end{multline*} 
The equality implies that $Z$ is a Shimura subvariety of $\sM$ of Hodge type for 
$\SO(d,2)$.\vspace{.1cm}
\item[ ii)] If $\sM$ is of type $\SO(n,2)$ and if the Griffiths-Yukawa coupling $\theta_{\bar Z}^2$ is zero then
\begin{multline*}\hspace*{1cm}
(d+1) \cdot \deg_{\omega_{\bar Z}(S_{\bar Z})}(\check{N}_{\bar Z/\bsM})+ (n-d-1)\cdot \deg_{\omega_{\bar Z}(S_{\bar Z})}(\Omega^1_{\bar Z}(\log S_{\bar Z})) =\\ n\cdot\big( \deg_{\omega_{\bar Z}(S_{\bar Z})}(\Omega^1_{\bar Z}(\log S_{\bar Z})) - (d+1)\cdot \deg_{\omega_{\bar Z}(S_{\bar Z})}(E^{2,0}_{\bar Z})\big)\geq 0.
\end{multline*} 
The equality implies that $Z$ is either the the deformation of a Shimura curve in $\sM$ or, if $\dim(Z)>1$, that $Z$ is a Shimura subvariety of $\sM$ of Hodge type for $\SU(d,1)$.
\vspace{.1cm}
\end{enumerate}
\end{theorem}
As in Remark \ref{weight2} on a Shimura variety $\sM$ of type $\SU(n,1)$ we consider the
weight two variation of Hodge structures with logarithmic Higgs bundle $(E_\bsM,\theta_\bsM)$ 
given as the direct sum of $(H,\tau)$ and its dual.
\begin{addendum}\label{HHPAD} \ 
\begin{enumerate}
\item[ iii)] If $\sM$ is of type $\SU(n,1)$, then the Griffiths-Yukawa coupling $\theta_{\bar Z}^2$ is zero and
\begin{multline*}\hspace*{1cm}
(d+1) \cdot \deg_{\omega_{\bar Z}(S_{\bar Z})}(\check{N}_{\bar Z/\bsM})+ (n-d)\cdot \deg_{\omega_{\bar Z}(S_{\bar Z})}(\Omega^1_{\bar Z}(\log S_{\bar Z})) =\\ (n+1)\cdot\big( \deg_{\omega_Z(S)}(\Omega^1_{\bar Z}(\log S_{\bar Z})) - (d+1)\cdot \deg_{\omega_{\bar Z}(S_{\bar Z})}(E^{2,0}_{\bar Z})\big)\geq 0.
\end{multline*} 
Again the equality implies that $Z$ is either the deformation of a Shimura curve in $\sM$ or, if $\dim(Z)>1$, that $Z$ is a Shimura subvariety of $\sM$ of Hodge type for $\SU(d,1)$.
\end{enumerate}
\end{addendum}
\begin{proof}[Proof of Theorem \ref{HHP} and Addendum \ref{HHPAD}]  All the arguments will concern $\bar Z$,
so for simplicity we will drop the lower index ${}_{\bar Z}$ for the Higgs bundles on $\bar Z$ and we will write
$\deg$ and $\mu$ instead of  $\deg_{\omega_{\bar Z}(S_{\bar Z})}$ and $\mu_{\omega_{\bar Z}(S_{\bar Z})}$.

Let us first show the equalities on the left hand sides. 
We know that 
$$
\varphi^*T_{\bsM}(-\log S_{\bsM})= E^{1,1}_\diamond \otimes E^{0,2}
$$ 
is a direct factor of $E^{1,1}\otimes E^{0,2}$. In Theorem \ref{HHP} both coincide and $\deg(E^{1,1})=0$. 
The exact sequence \eqref{taut} together with \ref{decomposition} gives then the equality
$$
-\deg(\Omega^1_{\bar Z}(\log S_{\bar Z})) +
\deg(\check{N}_{\bar{Z}/\bsM}) = n \cdot \deg(E^{0,2})= - n \cdot \deg(E^{2,0}),
$$
as claimed in i) and ii). 

For the Addendum we use the description of the Higgs bundle
of the weight two variation of Hodge structures $\W$ on $\sM$ in Remark \ref{weight2}. It is the direct sum
of two sub Higgs bundles, one in bidegree $(2,0)$ and $(1,1)$, the other in bidegrees $(1,1)$ and $(0,2)$.
So $(E,\theta)$ is the sum of $\varphi^*(H,\tau)$ and $\varphi^*(H^\vee,\tau^\vee)$, with $E^{2,0}=\varphi^*H^{1,0}$ invertible and with $E^{1,1}=\varphi^*H^{0,1}\oplus \varphi^*{H^{0,1}}^\vee$. Here 
$$
E^{1,1}_\diamond\otimes {E^{2,0}}^{-1} =\varphi^*(H^{0,1}\otimes {H^{1,0}}^{-1}) \cong \varphi^*T_{\bsM}(-\log S_{\bsM}).
$$
Since $(H,\tau)$ is the Higgs bundle of a local systems on $\sM$, its first Chern class is zero. 
The rank of $H^{0,1}$ is $n$ and therefore
\begin{multline*}
\deg(\varphi^* T_{\bsM}(-\log S_{\bsM})) = \deg(\varphi^*H^{0,1}) - n\cdot \deg(\varphi^*H^{1,0})=\\
-(n+1)\cdot \deg(\varphi^*H^{1,0})=
-(1+n) \cdot \deg(E^{2,0}).
\end{multline*}
The exact sequence \eqref{taut} together with \ref{decomposition} implies that
$$
-\deg(\Omega^1_{\bar Z}(\log S_{\bar Z})) + \deg(\check{N}_{\bar{Z}/\bsM})= -(n+1)\cdot \deg(E^{2,0}),
$$
hence the left hand equality in the Addendum \ref{HHPAD}.\\[.2cm]
The method to obtain the inequality and the interpretation of the extremal case is parallel to the one used in \cite{stz03} for the case $\dim(Z)=1$:\\[.2cm]
i) Consider the largest saturated Higgs subbundle $(F,\theta)$ of $(E,\theta)$ containing $E^{2,0}$. Hence writing as in 
\cite[Definition 1.7]{vz05} ${\rm Im'}$ for the saturated image, we get
$$
(F,\theta)=(E^{2,0}\oplus {\rm Im'}(E^{2,0}\otimes T_{\bar Z}(-\log S_{\bar Z}))\oplus{\rm Im'}(E^{2,0}\otimes S^2T_{\bar Z}(-\log S_{\bar Z})),\theta).
$$
The description of the Higgs bundle in Lemma \ref{pullback} b) implies that the saturated image 
${\rm Im'}(E^{2,0}\otimes S^2T_{\bar Z}(-\log S_{\bar Z}))$ is non-zero, hence it is isomorphic to
$E^{0,2}$. 

By Simpson \cite{sim92} $(E,\theta)$ is a $\mu$-polystable Higgs bundle and therefore
\begin{multline*}
\deg(E^{2,0}) +  \deg({\rm Im'}(E^{2,0}\otimes T_{\bar Z}(-\log S_{\bar Z})))\\
+ \deg({\rm Im'}(E^{2,0}\otimes S^2T_{\bar Z}(-\log S_{\bar Z})))=\
\deg(F) \leq 0.
\end{multline*}
Since $F^{2,0}=E^{2,0}$ and $F^{0,2}=E^{0,2}$ are dual to each other $\deg(F^{1,1})=\deg(F) \le 0$. 

The morphism $\varphi:Z\to \sM$ is generically finite over its image, hence the natural inclusion
$T_{\bar Z}(-\log S_{\bar Z})\to \varphi^* T_{\bsM}(-\log S_{\bsM})$ is injective and 
\begin{equation}\label{eqiso}
\xymatrix{
\theta: E^{2,0}\otimes T_{\bar Z}(-\log S_{\bar Z})\ar@{^{(}->}[r] & {\rm Im'}(E^{2,0}\otimes T_{\bar Z}(-\log S_{\bar Z}))}
\end{equation}
is an isomorphism over some open dense subscheme. Since $\omega_{\bar Z}(S_{\bar Z})$ is nef, this implies that $\deg( E^{2,0} \otimes T_{\bar Z}(-\log S_{\bar Z})) \le \deg F^{1,1} \le 0$. 
From this we obtain the Arakelov inequality 
$$
\deg(E^{2,0}) \le - \mu(T_{\bar Z}(-\log S_{\bar Z}))=\frac{\deg(\Omega^1_{\bar Z}(\log S_{\bar Z}))}{d}
$$ 
stated in i). Assume now, that this is an equality. Since $\omega_{\bar Z}(S_{\bar Z})$ is nef and ample with respect to $Z$, this forces the inclusion in \eqref{eqiso} to be an isomorphism on $Z$. In particular the two sheaves
$$
E^{2,0}\otimes T_{\bar Z}(-\log S_{\bar Z})\mbox{ \ \ and \ \ } F^{1,1}={\rm Im'}(E^{2,0}\otimes T_{\bar Z}(-\log S_{\bar Z}))
$$
are $\mu$-equivalent, as defined in \cite[Definition 1.7]{vz05}.

This  equality also implies that  
$$
\deg(E^{0,2})=\deg({\rm Im'}(E^{2,0}\otimes S^2T_{\bar Z}(-\log S_{\bar Z}))) = \deg(E^{2,0})+2\mu(T_{\bar Z}(-\log S_{\bar Z})).
$$
By Yau's Uniformization Theorem \cite{ya93} the sheaf $S^2 \Omega^1_{\bar Z}(\log S_{\bar Z})$ is $\mu$-polystable, and hence the saturated image of
\begin{equation}\label{image}
E^{2,0}\otimes S^2T_{\bar Z}(-\log S_{\bar Z})\longrightarrow E^{0,2}
\end{equation}
has to be $\mu$-equivalent to one of the direct factors. Again, the ampleness of $\omega_{\bar Z}(S_{\bar Z})$ with respect to $Z$ implies that the  morphism in \eqref{image} is surjective over $Z$. 

By \cite[Proposition 2.4]{vz05} we are allowed to apply Simpson's Higgs polystability, proven in
\cite{sim92}, although the slopes are taken with respect to a non-ample invertible sheaf. Since $F\subset E$
is $\mu$-equivalent to its saturated image and of degree zero, since $E^{2,0}=F^{2,0}$ and
$E^{0,2}=F^{0,2}$, one gets a direct sum decomposition
$$
(E,\theta)=(F,\theta)\oplus (U^{1,1},0)
$$ 
of Higgs bundles. The orthogonality of the splitting with respect to the Hodge metric implies that $(U^{1,1},0)$ comes from a unitary local system. By \ref{decomposition} $Z$ is  a subvariety of $\SO(d, 2)$ of Hodge type.\\[.2cm]
ii) The proof is similar. Here the saturated Higgs subsheaf $(F,\theta)$ generated by $E^{2,0}$ is given by
$$
(F, \theta)=(E^{2,0}\oplus {\rm Im'}(E^{2,0}\otimes T_{\bar Z}(-\log S_{\bar Z})),\theta).
$$ 
Then
$$
\deg(E^{2,0}) +  \deg({\rm Im'}(E^{2,0}\otimes T_{\bar Z}(-\log S_{\bar Z})))\leq 0
$$
and the corresponding Arakelov inequality says
$$ 
\deg E^{2,0}\leq \frac{\deg\Omega^1_{\bar Z}(\log S_{\bar Z})}{d+1}.
$$
The equality holds if and only there is a decomposition
$$ (E,\theta)=(F,\theta)\oplus (F,\theta)^\vee\oplus (U^{1,1},0),$$
such that 
$$\theta: E^{2,0}\otimes T_{\bar Z}(-\log S_{\bar Z})\to {\rm Im'}(E^{2,0}\otimes T_{\bar Z}(-\log S_{\bar Z}))$$
is an isomorphism over $Z$, hence a $\mu$-equivalence. Again $(U^{1,1},0)$ is the Higgs bundle of a unitary local system in this case, and by \ref{decomposition} $Z$ is a Shimura subvariety of Hodge type
for $\SU(d,1)$ if $d\geq 2$ or a deformation of such for $d=1$.\\[.2cm]
iii) Finally let $\sM$ be a Shimura variety of $\SU(n,1)$-type. Using the notation from Remark \ref{weight2}
the uniformizing Higgs bundle of weight one has the Higgs field
$$
\begin{CD}
\tau: T_{\sM}\otimes H^{1,0} @> \simeq >> H^{0,1}. 
\end{CD}
$$
In \cite{vz05} we proved the Arakelov inequality, saying that 
\begin{equation}\label{ara1}
(d+1)\cdot \deg{E^{2,0}}=(d+1)\cdot \deg(\varphi^*(H^{1,0})) \leq \deg(\Omega^1_{\bar Z}(\log S_{\bar Z})),
\end{equation}
and that the equality forces $Z$ to be a Shimura subvariety of Hodge type for $\SU(d,1)$.
In the present situation the proof is quite simple.
Let $(H_{\bar Z},\tau_{\bar Z})$ denote the Higgs field on $\bar Z$ induced by the Deligne extension.
Take the sub Higgs sheaf $(F,\theta)$  generated by $H^{1,0}_{\bar Z}$. Again Simpson shows that the degree of  
$$(F,\theta)=(E^{1,0}\oplus {\rm Im'}(T_{\bar Z}(-\log S_{\bar Z})\otimes E^{1,0\vee}),\theta)$$
is non-positive, hence that \eqref{ara1} holds. 

The equality implies that $Z\subset \sM$ is totally geodesic. Since $\sM$ is of type $\SU(n,1)$ the sheaf
$\Omega^1_{\bsM}(\log S_\bsM)$ is ample with respect to $\sM$. Then the subvariety $Z\subset \sM$ is rigid.
Hence by Theorem~\ref{AbdMo} $Z$ is a Shimura subvariety of Hodge type for $\SU(d,1)$.
\end{proof}

\begin{remark}\label{HHPRC}
We say that the Hirzebruch-H\"ofer proportionality (HHP) holds, if the inequalities\vspace{.1cm}
\begin{enumerate}
\item[ i)] $\mu_{\omega_{\bar Z}(S_{\bar Z})}(\check{N}_{\bar Z/\bsM}) \geq \mu_{\omega_{\bar Z}(S_{\bar Z})}(T_{\bar Z}(-\log S_{\bar Z}))$\vspace{.1cm}
\item[ ii)] 
$(d+1) \cdot \deg_{\omega_{\bar Z}(S_{\bar Z})}(\check{N}_{\bar Z/\bsM})\geq (n-d-1)\cdot \deg_{\omega_{\bar Z}(S_{\bar Z})}(T_{\bar Z}(-\log S_{\bar Z}))$\vspace{.1cm}
\item[ iii)] $
(d+1) \cdot \deg_{\omega_{\bar Z}(S_{\bar Z})}(\check{N}_{\bar Z/\bsM})\geq (n-d)\cdot \deg_{\omega_{\bar Z}(S_{\bar Z})}(T_{\bar Z}(-\log S_{\bar Z}))$\vspace{.1cm}
\end{enumerate}
in Theorem \ref{HHP} i), ii) and in the Addendum \ref{HHPAD} iii) are equalities. 

If $Z$ is a divisor in $\sM$, hence $n=d+1$, then the HHP in Theorem \ref{HHP}, ii), just says that
the degree of the logarithmic normal sheaf is non-negative, and that $Z$ is a Shimura subvariety of Hodge type, if and only if it is zero. 
\end{remark}
\begin{proof}[Proof of Theorem \ref{HHP0}]
For $\bar{Y}=\bsM$ and for a non-singular curve $\bar{C}=\bar{Z}\subset \bar{Y}$ the equality of Chern numbers \eqref{tautchern} gives $\deg(\check{N}_{{\bar C}/{\bar Y}}) = \bar{C}.\bar{C} + S_{\bar{C}}-S_{\bar{Y}}.\bar{C}$. 
Since $n=2$ and $d=1$ the inequality i) in \ref{HHPRC} says that
$$
\bar{C}.\bar{C}+\deg(S_{\bar{C}})-S_{\bar{Y}}.\bar{C}\geq 
\deg(-K_{\bar{C}}-S_{\bar{C}})=(-K_{\bar Y}-\bar{C}).\bar{C}-
\deg(S_{\bar{C}}),
$$
as stated in \eqref{eqp1}. The inequality iii) translates to
$$
2\cdot \bar{C}.\bar{C}+2\cdot \deg(S_{\bar{C}})-2\cdot S_{\bar{Y}}.\bar{C}\geq  
(-K_{\bar Y}-\bar{C}).\bar{C}-
\deg(S_{\bar{C}}),
$$
hence to \eqref{eqp2}. 
\end{proof}
The remaining inequality ii) is the additional inequality \eqref{eqp3}. However, as explained in the introduction, the assumptions made for ii) imply that $Y$ is the product of two curves and $C$ one of the fibres,
so $\bar{C}.\bar{C}=0$.

\section{Subvarieties of $\sM$  containing Special subvarieties}\label{SSV}

From now on $\sM$ will be a Shimura variety of type $\SO(n,2)$. We consider a closed subvariety $Z\subset \sM$, and we study subvarieties 
$W\subset Z$ which are Shimura subvarieties of $\sM$ of Hodge type. 
We hope that the existence of sufficiently many of them forces $Z$ to be itself a Shimura subvariety of Hodge type. In Section \ref{char} we will see, that this hope is fulfilled if their codimension in $Z$ is one.

\begin{assumptions}\label{assSect3}
Consider a projective manifold $\bar Z$ and the complement $Z$ of a strict normal crossing divisor $S_{\bar Z}$. Assume one has generically finite morphisms
$$
\begin{CD}
W @> \psi >>  Z @> \varphi >> \sM \mbox{ \ \ and \ \ } \sigma=\varphi\circ \psi,
\end{CD}
$$
such that $\sigma(W)$ is not contained in the singular locus of $\varphi(Z)$.
We write 
$$
n=\dim(\sM),  \ \ \ d=\dim(Z)\mbox{ \ \ and \ \ }m=\dim(W).
$$
Assume that $\sM$ is a Shimura variety of type $\SO(n,2)$, that $W$ is the deformation of
a Shimura subvariety of Hodge type, and that $\sigma$ is induced by a morphism of groups. In particular its image is non-singular and $\sigma$ is \'etale over the image.
We choose Mumford compactifications $\bsM=\sM \cup S_\bsM$ and $\bar W= W \cup S_{\bar W}$ and write
again 
$$
\begin{CD}
{\bar W} @> \psi >>  {\bar Z} @> \varphi >> \bsM
\end{CD}
$$
for the induced rational maps. We keep the assumption that the uniformizing variation of Hodge structures
$\V$ on $\sM$ has unipotent local monodromy at infinity, and we decompose $\V$ as a direct sum $\W\oplus \U$, where $\U$ is the largest unitary subvariation of Hodge structures.  
$$
(E_{\bar W},\theta_{\bar W}), \ \ (E_{\bar Z},\theta_{\bar Z})\mbox{ \ \ and \ \ }(E_{\bsM},\theta_{\bsM})
$$ 
denote the Higgs bundles induced by the Deligne extension of $\sigma^*\W$, $\varphi^*\W$ and $\W$.
\end{assumptions}
Recall that by \ref{del} the pullbacks under rational maps are just the reflexive hulls of the pullback to the largest open subscheme, where the morphisms are defined. In particular one has on $\bar W$ the 
natural maps
\begin{equation}\label{taut1} 
 T_{\bar W}(-\log S_{\bar W}) \longrightarrow \psi^*T_{\bar Z}(-\log S_{\bar Z}) \longrightarrow \sigma^*T_{\bsM}(-\log S_\bsM).
\end{equation}
By Lemma \ref{pullback} the sheaf $\sigma^*T_{\bsM}(-\log S_\bsM)$ is locally free, whereas $\psi^*T_{\bar Z}(-\log S_{\bar Z})$ is just torsion-free. Since $\sigma(W)$ meets the non-singular locus of $\varphi(Z)$
both morphisms in \eqref{taut1} are injective. We define again the logarithmic normal sheaf by the exact sequence
\begin{equation}\label{taut2}
0 \longrightarrow T_{\bar W}(-\log S_{\bar W}) \longrightarrow  \psi^*T_{\bar Z}(-\log S_{\bar Z})\longrightarrow \check{N}_{\bar W/\bar Z}\longrightarrow 0. 
\end{equation}
\begin{lemma}\label{torsionfree} \ 
\begin{enumerate}
\item[a.] The sheaf $\check{N}_{\bar W/\bar Z}$ is torsion-free and the exact sequence \eqref{taut2}
splits.
\item[b.] Let $\check{N}^\natural_{\bar W/\bar Z}$ be the saturated hull
of $\check{N}_{\bar W/\bar Z}$ in $\sigma^*T_{\bsM}(-\log S_\bsM)$, i.e. 
$$
\check{N}^\natural_{\bar W/\bar Z}={\rm Ker}\big[\sigma^*T_{\bsM}(-\log S_\bsM)\to \big(\sigma^*T_{\bsM}(-\log S_\bsM)/_{\check{N}_{\bar W/\bar Z}}\big)/_{\rm torsion}\big].
$$
Then $\mu_{\omega_{\bar W}(S_{\bar W})}(\sF)(\check{N}_{\bar W/ \bar Z})\leq \mu_{\omega_{\bar W}(S_{\bar W})}(\check{N}^\natural_{\bar W/ \bar Z})$.
\end{enumerate}
\end{lemma}
\begin{proof} Since $\sigma: W \to \sM$ maps to a Shimura subvariety of Hodge type, or to a deformation of such a variety, we are allowed to apply Lemma \ref{pullback} d). So there is a surjection
$\eta:\sigma^*T_{\bsM}(-\log S_\bsM)\to T_{\bar W}(-\log S_{\bar W})$ whose restriction to the
subsheaf $T_{\bar W}(-\log S_{\bar W})$ is an isomorphism. So the restriction of $\eta$ to $\psi^*T_{\bar Z}(-\log S_{\bar Z})$ defines 
a splitting of this sheaf as well.

The sheaf $\check{N}_{\bar W/\bar Z}$ is contained in $\check{N}_{\bar W/\bsM}$ and
by Lemma \ref{pullback} the latter is locally free. Part b) follows since $\omega_{\bar W}(S_{\bar W})$ is nef. 
\end{proof}
\begin{theorem}\label{1.6} Under the Assumptions made in \ref{assSect3} one has:
\begin{enumerate}
\item[ i)] If the Griffiths-Yukawa coupling is non-zero on $W$, then
\begin{equation*}
\mu_{\omega_{\bar W}(S_{\bar W})}(\check{N}^\natural_{\bar W/ \bar Z})\leq \mu_{\omega_{\bar W}(S_{\bar W})}(T_{\bar W}(-\log S_{\bar W}))<0.
\end{equation*}
\item[ ii)]  If the Griffiths-Yukawa coupling on $Z$ vanishes, then
\begin{equation*}
\frac{\deg_{\omega_{\bar W}(S_{\bar W})}(\check{N}^\natural_{\bar W/\bar Z})}{\rk \check{N}_{\bar W/\bar Z}} \leq \frac{\deg_{\omega_{\bar W}(S_{\bar W})}(T_{\bar W}(-\log S_{\bar W}))}{m+1}<0.
\end{equation*}
\item[iii)] Assume  that $W \to  \sM$ is the deformation of a Shimura curve of Hodge type, and that $Z\subset\sM$ is a quasi-projective surface.  Then
\begin{equation*}
\deg_{\omega_{\bar W}(S_{\bar W})} \check{N}^\natural_{\bar W/\bar Z}\leq 0.
\end{equation*}
If this is an equality, then the Griffiths-Yukawa coupling along $W$ vanishes and the Griffiths-Yukawa coupling on $Z$ does not vanish.
\end{enumerate}
\end{theorem}
In order to state what happens if the inequalities in Theorem \ref{1.6} are equalities, we need 
some more notations. Recall that the Higgs field
$$
\theta:  T_{\bsM}(-\log S_\bsM)\otimes  E^{2,0}_\bsM\longrightarrow E^{1,1}_\bsM
$$
is an isomorphism. Consider the tautological sequence for $\varphi:{\bar Z} \to \bsM$  
$$
0 \longrightarrow  T_{\bar Z}(-\log S_{\bar Z})\longrightarrow \varphi^*T_{\bsM}(-\log S_\bsM)\longrightarrow \check{N}_{{\bar Z}/\bsM} \longrightarrow 0.
$$
Via the identification $T_{\bsM}(-\log S_\bsM)\otimes  E^{2,0}_\bsM= E^{1,1}_\bsM$ the inclusion 
$$T_{\bar Z}(-\log S_{\bar Z})\longrightarrow \varphi^*T_{\bsM}(-\log S_\bsM)$$ 
tensorized with ${\rm id}_{E^{2,0}_{\bar Z}}$ is 
$\theta: T_{\bar Z}(-\log S_{\bar Z})\otimes E^{2,0}_{\bar Z} \longrightarrow E^{1,1}_{\bar Z}$.
 
We now consider the saturated Higgs subsheaf $(F,\theta) \subset (E_{\bar Z},\theta_{\bar Z})$, which is generated by $E^{2,0}_{\bar Z}$. So one has
\begin{gather}\label{FHiggs1}
F^{2,0}=E^{2,0}_{\bar Z}, \ \ \  F^{1,1}={\rm Im'}\big(\theta:T_{\bar Z}(-\log S_{\bar Z})\to E^{1,1}_{\bar Z}\big),\\
\label{FHiggs2}
\mbox{and \ \ } F^{0,2}={\rm Im'}\big(\theta^2: S^2T_{\bar Z}(-\log S_{\bar Z})\to E^{0,2}_{\bar Z}\big).
\end{gather}
In particular $F^{0,2}$ is zero if the Griffiths-Yukawa coupling is zero, and equal to $E^{0,2}$ otherwise. 
\begin{addendum}\label{3.1} \ 
In Theorem \ref{1.6} assume that:
\begin{gather}\label{eqi}
\hspace{2cm}\mu_{\omega_{\bar W}(S_{\bar W})}(\check{N}_{\bar W/ \bar Z})=\mu_{\omega_{\bar W}(S_{\bar W})}(T_{\bar W}(-\log S_{\bar W})) \hspace{1.3cm} \mbox{ in case i),}\\
\label{eqii}
\hspace{1.7cm}
\frac{\deg_{\omega_{\bar W}(S_{\bar W})}(\check{N}_{\bar W/\bar Z})}{\rk \check{N}_{\bar W/\bar Z}} = \frac{\deg_{\omega_{\bar W}(S_{\bar W})}(T_{\bar W}(-\log S_{\bar W}))}{m+1}\hspace{.8cm} \mbox{ in case ii),}\\
\label{eqiii}\hspace{2.2cm}
\deg_{\omega_{\bar W}(S_{\bar W})} \check{N}_{\bar W/\bar Z}= 0 \hspace{5.4cm} \mbox{ in case iii).}
\end{gather}
Then $\psi^*{\rm Im}\big(\theta:T_{\bar Z}(-\log S_{\bar Z})\to E^{1,1}_{\bar Z}\big)$ and $\psi^*F^{1,1}$ are torsionfree. The inclusions
$$
\psi^*{\rm Im}\big(\theta:T_{\bar Z}(-\log S_{\bar Z})\to E^{1,1}_{\bar Z}\big) \longrightarrow\psi^*F^{1,1} \longrightarrow (\psi^*F^{1,1})^\natural
$$ 
are isomorphisms on $W$ and $\mu_{\omega_{\bar W}(S_{\bar W})}$ equivalences.
Moreover $(\psi^* F)^\natural$ is a direct factor of
$E_{\bar W}$ and hence locally free and $\ch_1((\psi^*F)^\natural)=0$.
In particular on $W$ the sequence 
$$
0 \longrightarrow  \psi^*T_{Z}\longrightarrow \sigma^*T_{\sM}\longrightarrow \psi^* N_{{Z}/\sM} \longrightarrow 0
$$
remains exact and splits.
\end{addendum}

Recall that an inclusion $\sF \subset \sG$ is a $\mu_{\omega_{\bar W}(S_{\bar W})}$-equivalence, if
both sheaves have the same rank, and if $\mu_{\omega_{\bar W}(S_{\bar W})}(\sF)=\mu_{\omega_{\bar W}(S_{\bar W})}(\sG)$.
Since $\omega_{\bar W}(S_{\bar W})$ is nef and ample with respect to $W$ this implies that $\sF \to \sG$
is an isomorphism over $W$.

The statement of Addendum \ref{3.1} says that infinitesimally in a neighborhood of $W$ the subscheme
$Z$ of $\sM$ looks like a Shimura subvariety. In the next section we will show that such an information
for sufficiently many divisors forces $Z$ to be a Shimura variety.
\begin{proof}[Proof of Theorem \ref{1.6} and of the Addendum \ref{3.1}] \ \\[.1cm]
For simplicity from now on slopes and degrees will always be with respect to $\ch_1(\omega_{\bar W}(S_{\bar W}))$, so we drop the lower index and write $\deg(\sF)$ and $\mu(\sF)$ instead of $\deg_{\omega_{\bar W}(S_{\bar W})}(\sF)$ and $\mu_{\omega_{\bar W}(S_{\bar W})}(\sF)$. 

We will also write $(E,\theta)$ instead of $(E_{\bar W},\theta_{\bar W})$ for the Higgs field induced by the  Deligne extension of the variation of Hodge structures to $\bar W$. In order to prove
Theorem \ref{1.6} and Addendum \ref{3.1} we will consider each of the cases i), ii) and iii) separately.\\[.2cm]
Let us start with case i).
Since $W$ is a subvariety of Hodge type and since the Griffiths-Yukawa coupling is non-zero,
there is a decomposition of the form 
$$(E^{2,0}\oplus E^{1,1}\oplus E^{0,2},\theta)|_W=(E^{2,0}\oplus E^{1,1}_\diamond\oplus E^{0,2},\theta)|_W\oplus (U^{1,1},0)|_W,$$ 
such that the first component uniformizes $W$ and such that the second component is the Higgs field of a unitary variation of Hodge structures of bidegree $(1,1)$. The compatibility of the Deligne extension with pullbacks implies that this decomposition extends to $\bar W$. 
 
Since $E^{1,1}_\diamond\simeq T_{\bar W}(-\log S_{\bar W} )\otimes E^{2,0}$ is $\omega_{\bar W}(S_{\bar W})$-polystable of slope zero and since $U^{1,1}$ corresponds to a unitary local system over $\bar W$, we see that
$$\sigma^*T_{\bsM}(-\log S_\bsM)\simeq E^{1,1}\otimes E^{0,2}=(E^{1,1}_\diamond\oplus U^{1,1})\otimes E^{0,2}\simeq T_{\bar W}(-\log S_{\bar W})\oplus U^{1,1}\otimes E^{0,2}$$ is
$\omega_{\bar W}(S_{\bar W})$-polystable and that $T_{\bar W}(-\log S_{\bar W})$ is a direct factor of
$\sigma^*T_{\bsM}(-\log S_\bsM)$, hence of $\psi^*T_{\bar Z}(-\log S_{\bar Z})$.

By the exact sequence \eqref{taut2} the projection from $\sigma^*T_{\bsM}(-\log S_\bsM)$
to $U^{1,1}\otimes E^{0,2}$ induces an injection 
$$
\check{N}^\natural_{\bar W/\bar Z} \longrightarrow U^{1,1}\otimes E^{0,2}.
$$
So 
$$
\mu(\check{N}^\natural_{\bar W/\bar Z})\leq \mu(U^{1,1}\otimes E^{0,2})=
\mu(T_{\bar W}(-\log S_{\bar W})),
$$
as stated in Part i). 

By Lemma \ref{torsionfree} the sheaf $F^{1,1}$ is without torsion. Since
$$
\mu(\check{N}_{\bar W/\bar Z}) \leq 
\mu(\check{N}^\natural_{\bar W/\bar Z}) \leq 
\mu(U^{1,1}\otimes E^{0,2}),
$$
the equality \eqref{eqi} shows that $\check{N}_{\bar W/\bar Z}$ and $\check{N}^\natural_{\bar W/\bar Z}$ are $\mu$-equivalent. Moreover, as explained in \cite[Proposition 2.4]{vz05}
the Simpson correspondence \cite{sim92} implies that $\check{N}^\natural_{\bar W/\bar Z}\otimes E^{2,0}$ is a direct factor of $U^{1,1}$. 

The saturated hull $(\psi^*F^{1,1})^\natural$ is nothing but $E^{1,1}_\diamond \oplus \check{N}^\natural_{\bar W/\bar Z}\otimes E^{2,0}$, hence of slope zero. Since the Griffiths-Yukawa coupling is non-zero one has $(\psi^*F^{0,2})= E^{0,2}$, and since $(\psi^*F^{2,0})=E^{2,0}$ one finds that $(\psi^*F)^\natural$ is a direct factor of $E$. 

Obviously $\psi^*{\rm Im}\big(\theta:T_{\bar Z}(-\log S_{\bar Z})\to E^{1,1}_{\bar Z}\big)$,  
$\psi^*F^{1,1}$ and $(\psi^*F^{1,1})^\natural$ are $\mu$-equivalent, so we verified Addendum \ref{3.1} in case i), except of the strict inequality on the right hand side.\\[.2cm]
Before finishing i), let us consider the case where the Griffiths-Yukawa coupling on $W$ is zero. This holds in iii) by assumption, and in ii) since the Griffiths-Yukawa coupling on $Z$ is zero. 
Then one obtains on $\bar W$ a different type of decomposition,
$$ (E^{2,0}\oplus E^{1,1}\oplus E^{0,2},\theta)=(E^{2,0}\oplus E^{1,1}_\diamond,\theta)\oplus (U^{1,1},0) \oplus (E^{1,1\vee}_\diamond\oplus E^{0,2},\theta).$$ 
Here $(E^{2,0}\oplus E^{1,1}_\diamond,\theta)$ uniformizes $W$ as a ball quotient, $(U^{1,1},0)$ is the Higgs bundle of a unitary variation of Hodge structures of type $(1,1)$, and $ (E^{1,1\vee}_\diamond\oplus E^{0,2},\theta)$ is the dual of $(E^{2,0}\oplus E^{1,1}_\diamond,\theta)$.  

Note that the uniformization gives $T_{\bar W}(-\log S_{\bar W})\simeq E^{1,1}_\diamond\otimes E^{0,2}$.
Hence, one has
$$\sigma^*T_{\bsM}(-\log S_\bsM)\simeq E^{1,1}\otimes E^{0,2}=(E^{1,1}_\diamond\oplus U^{1,1}\oplus E_\diamond^{1,1\vee})\otimes E^{0,2}$$
$$\simeq T_{\bar W}(-\log S_{\bar W})\oplus U^{1,1}\otimes E^{0,2}\oplus E^{1,1\vee}_\diamond\otimes E^{0,2}$$
and contrary to the case i) $\sigma^*T_{\bsM}(-\log S_\bsM)$ is not polystable. Nevertheless the sheaf $T_{\bar W}(-\log S_{\bar W})$ is a direct factor of $\sigma^*T_{\bsM}(-\log S_\bsM)$ and of
$\psi^*T_{\bar Z}(-\log S_{\bar Z})$.
 
Dividing by $E^{1,1}_\diamond\otimes E^{0,2}$
the exact sequence \ref{taut2} defines an embedding
\begin{equation}\label{ii+iii}
\check{N}^\natural_{\bar W/\bar Z}\hookrightarrow  U^{1,1}\otimes E^{0,2}\oplus E^{1,1\vee}_\diamond\otimes E^{0,2}
\end{equation}
and
$(\psi^*F^{1,1})^\natural=E^{1,1}_\diamond \oplus \check{N}^\natural_{\bar W/\bar Z}\otimes E^{2,0}$.\\[.2cm]
In ii) we assumed that the Griffiths-Yukawa coupling vanishes on $Z$. 
So the image of $\check{N}^\natural_{\bar W/\bar Z}$ must lie in the kernel of 
$$
\theta\otimes {\rm id}_{E^{0,2}} : E^{1,1} \otimes E^{0,2} \longrightarrow {E^{0,2}}^{\otimes 2} \otimes \Omega^1_{\bar Z}(\log S_{\bar Z}),
$$ 
hence in $U^{1,1}\otimes E^{0,2}$. Since the sheaf $U^{1,1}\otimes E^{0,2}$ is $\omega_{\bar W}(S_{\bar W})$-polystable of slope $\deg(E^{0,2})$, we obtain
\begin{equation}\label{eqiib}
\mu(\check{N}_{\bar W/\bar Z})\leq \mu(U^{1,1}\otimes E^{0,2})=\deg(E^{0,2}).
\end{equation}
Applying Theorem \ref{HHP}, ii), to the pair $W\to \sM$ one obtains 
$$
\ch_1(T_{\bar W}(-\log S_{\bar W}))=(\dim(W)+1)\cdot \ch_1(E^{0,2}),
$$
and therefore
$$
\mu(\check{N}^\natural_{\bar W/\bar Z}) \leq\frac{\deg(T_{\bar W}(-\log S_{\bar W}))}{\dim (W)+1},
$$
as stated in Part ii).

The equality \eqref{eqiib} implies that this is an equality and that
$\check{N}_{\bar W/\bar Z}\to \check{N}^\natural_{\bar W/\bar Z}$ is a $\mu$-equivalence.
Moreover one finds 
\begin{multline*}
\deg((\psi^*F)^\natural) = \deg(E^{2,0} \oplus E^{1,1}_\diamond \oplus \check{N}^\natural_{\bar W/\bar Z}\otimes E^{2,0})=\\ (d+1)\cdot \deg(E^{2,0})+
\deg(T_{\bar W}(-\log S_{\bar W}))+ \deg(\check{N}^\natural_{\bar W/\bar Z})=\\
(d+1 - (m+1) - (d-m))\cdot \deg(E^{2,0})=0. 
\end{multline*}
By \cite[Proposition 2.4]{vz05} the Simpson correspondence implies that $(\psi^*F)^\natural \subset E$ 
is a direct factor. \\[.2cm]
For Theorem \ref{1.6}, i) and ii), it remains to verify the strict inequality on the right hand side.
In both cases the sheaf $\check{N}^\natural_{\bar W/\bar Z}$ is isomorphic to a subsheaf
of the semistable sheaf $U^{1,1}\otimes E^{0,2}$, hence $\mu(\check{N}^\natural_{\bar W/\bar Z}) \leq \mu(E^{0,2}) < 0$.\\[.2cm] 
In the remaining case iii) we have again the embedding in \eqref{ii+iii}. Since $W$ is a curve, using the notation introduced there, one has 
$$
E^{1,1\vee}_\diamond\otimes E^{0,2}=(E^{2,0}\otimes T_{\bar W}(-\log S_{\bar W}))^{\vee}\otimes E^{0,2}={\mathcal O}_{\bar W}
$$
and $\check{N}^\natural_{\bar W/\bar Z}$ is a subsheaf of $U^{1,1}\otimes T_{\bar W}(-\log S_{\bar W})^{1/2}\oplus {\mathcal O}_{\bar{W}}$. This inclusion implies that $\deg(\check{N}^\natural_{\bar W/\bar Z}) \leq 0$, as stated in Part iii).

If the equation \eqref{eqiii} holds one has $\check{N}_{\bar W/\bar Z}=\check{N}^\natural_{\bar W/\bar Z}$ and both are of degree zero. So the projection to the negative sheaf $U^{1,1}\otimes T_{\bar W}(-\log S_{\bar W})^{1/2}$ must be zero, hence $\check{N}_{\bar W/\bar Z} = {\mathcal O}_{\bar W}$. Then  
$$
\psi^*F^{1,1} = T_{\bar W}(-\log S_{\bar W})\otimes E^{2,0}\oplus E^{2,0}, 
$$ 
and since $F^{2,0}=E^{2,0}$ and $F^{0,2} = E^{0,2}$ one obtains
$\deg((\psi^*F)^\natural) = 0$. 
Again the Simpson correspondence implies that $(\psi^*F,\theta) \subset (E,\theta)$ 
is a direct factor.
\end{proof}

\section{A characterization of subvarieties of a Shimura variety $\sM$ of type $\SO(n,2)$.}\label{char}
 
In this section we start with an auxiliary result on a finite set of divisors $\{Y_i\}_{i\in I}$ on projective manifolds $X$. Later $X$ will be the compactification of a subscheme of $\sM$ and the $Y_i$ will be compactifications of Shimura subvarieties $W_i$ of $\sM$. 
 
\begin{lemma}\label{2.1}   Let $X$ be a smooth projective manifold of dimension $d$ and let $\{Y_i\}_{i\in I}$ be a set of pairwise distinct prime divisors. Let $\rho$ be the Picard number of $X$, let $A$ be a nef and big divisor on $X$ and assume that $Y_i^2.A^{d-2}<0$ for all $i\in I$.

If $\# I \geq \rho^2 + \rho + 1$ then there exists a linear combination $D=\sum_{i\in I} a_iY_i$ with 
$a_i \in \N$ and $D^2.A^{d-2} > 0$.
\end{lemma}

\begin{proof}  Let ${\rm NS}(X)_\Q$ be the $\Q$-Neron-Severi group of $X$ and let $\equiv$ stand for ``numerical equivalence''.
Remark first that for effective divisors $D$ and $D'$ without common components, the intersection
$D.D'$ is a linear combination of codimension two subschemes with non-negative coefficients. Since $A$ is nef, one 
obtains $D.D'.A^{d-2} \geq 0$.

We start with any subset $I_1\subset I$ of cardinality $\rho+1$, say $I_1=\{1,\ldots,\rho+1\}$.
The images of the divisors $Y_1,\ldots, Y_{\rho+1}$ in ${\rm NS}(X)_\Q$ must be linearly dependent, hence there exist $a_1,\ldots, a_l, b_1,\ldots, b_m \in {\N}$ with
$$
D_1\equiv D_1' \mbox{ \ \ for \ \ } D_1=\sum _{i=1}^la_i Y_i \mbox{ \ \ and \ \ } D'_1=\sum_{j=1}^m b_j Y_j.
$$
Since $D_1$ and $D_1'$ are effective divisors without common components
one obtains $D_1^2.A^{d-2}=D_1'D_1.A^{d-2} \geq 0$. If $D_1^2.A^{d-2}>0$, we are done. 

If $D_1^2.A^{d-2}=0$ and if there exists a divisor $Y_j$ with $j> \rho+1$ and with $D.Y_j.A^{d-2} > 0$,
then for $m$ sufficiently large $(mD_1+Y_j)^2.A^{d-2}>0$, and again we found the divisor we are looking for.

Hence if the statement of Lemma \ref{2.1} is wrong, for any system of disjoint subsets
$I_1,\ldots,I_{\rho}\subset I$ with $\# I_\iota = \rho+1$, we can find effective non-zero divisors
$$
D_\iota=\sum _{i\in I_\iota} a_i\cdot Y_i 
$$ 
with $D_\iota^2.A^{d-2}=0$ and with $Y_j.D_\iota.A^{d-2}=0$ for all $j\in I\setminus I_\iota$.
In particular, choosing $\nu_\iota \in I_\iota$ with $a_{\nu_\iota}\neq 0$ the intersection $Y_{\nu_\iota}.Y_j.A^{d-2}=0$ for $j\in I\setminus I_\iota$. 

By assumption $I\setminus (I_1 \cup \cdots \cup I_\rho)$ still contains one element, say $\nu_{\rho+1}$, and
the intersection of $A^{d-2}$ with two different divisors in $\{ Y_{\nu_1}, \ldots , Y_{\nu_{\rho+1}}\}$ is zero. 
So given a linear combination
$$
0=\sum_{\iota=1}^{\rho+1} \alpha_{\iota}\cdot Y_{\nu_\iota} \mbox{ \ \ one finds \ \ }
0= \sum_{\iota=1}^{\rho+1}\alpha_{\iota}\cdot Y_{\nu_\iota}.Y_{\nu_k}.A^{d-2} = \alpha_{k}\cdot Y_{\nu_k}.Y_{\nu_k}.A^{d-2},
$$
for all $1 \le k \le \rho+1$. The assumption $Y_i^2.A^{d-2}<0$ implies 
$$\alpha_{1}= \cdots = \alpha_{\rho+1}=0,$$
and hence the images of $\{ Y_{\nu_1}, \ldots , Y_{\nu_{\rho+1}}\}$ in ${\rm NS}(X)_\Q$ are linear independent, a contradiction.
\end{proof}

From now on, as indicated in the title of this section, $\sM$ will again be a Shimura variety of type $\SO(n,2)$. 
\begin{assumptions}\label{assSect4}
Consider a projective manifold $\bar Z$ of dimension $d\geq 2$ and the complement $Z$ of a strict normal crossing divisor $S_{\bar Z}$ with $\Omega^1_{\bar Z}(\log S_{\bar Z})$ nef and with $\omega_{\bar Z}(S_{\bar Z})$ ample with respect to $Z$. Given an injection $\varphi: Z \to \sM$ and a finite index
set $I$, consider for $i\in I$ non-singular irreducible divisors $W_i$ on $Z$, hence $m=\dim(W_i)=d-1$.

The corresponding embeddings will be denoted by 
$$
\begin{CD}
W_i @> \psi_i >>  Z @> \varphi >> \sM \mbox{ \ \ and \ \ } \sigma_i=\varphi\circ \psi_i.
\end{CD}
$$
We assume that $\sM$ is a Shimura varieties of type $\SO(n,2)$, that the $W_i$ are deformations of Shimura subvarieties of Hodge type, and that $\sigma_i$ is induced by a morphism of groups. 
Choosing Mumford compactifications $\bsM=\sM \cup S_\bsM$ and $\bar W_i=W_i \cup S_{{\bar W}_i}$, we write again 
$$
\begin{CD}
{\bar W}_i @> \psi_i >>  {\bar Z} @> \varphi >> \bsM\mbox{ \ \ and \ \ } \sigma_i=\varphi\circ \psi_i
\end{CD}
$$
for the induced rational maps. Let $\V$ be the uniformizing variation of Hodge structures on $\sM$, let $\U$ be the largest unitary subvariation of Hodge structures and $\V=\W\oplus \U$. As usual $(E_{\bar Z},\theta_{\bar Z})$ will denote the Higgs bundle induced by the Deligne extension of $\varphi^*\W$.
\end{assumptions}
\begin{notations}
Let $\rho$ denote the Picard number of $\bar Z$ and let $\delta$ denote the number of non-empty intersections
$S_\ell\cap S_k$ of different components $S_\ell$ and $S_k$ of $S_{\bar Z}$.
We define $\varsigma(\bar{Z})=\rho^2 + \rho + 1$ if one of the following conditions holds:
\begin{enumerate}\label{3.2n}
\item $d=\dim(Z)\geq 4$.
\item For all $i\in I$ the divisor $S_{\bar Z}|_{\psi_i({\bar W}_i)} - (S_{\bar Z}|_{\psi_i({\bar W}_i)})_{\rm red}$, considered in \eqref{tautchern}, is zero.
\item For all $i,j\in I$ one has $\psi_i(\bar{W}_i) \cap \psi_j(\bar{W}_j)\neq \emptyset$.
\end{enumerate}
Otherwise we choose $\varsigma(\bar{Z})=(\rho+\delta)^2 + \rho + \delta + 1$.
\end{notations}
\begin{theorem}\label{3.2}  Under the assumptions made in \ref{assSect4} one has:
\begin{enumerate}
\item[i)] If the $W_i$ are of type  $\SO(d-1,2)$, for all $i\in I$, if they satisfy the HHP equality  
$$
\mu_{\omega_{{\bar W}_i}(S_{{\bar W}_i})}(\check{N}_{{\bar W}_i/{\bar Z}})=
\mu_{ \omega_{{\bar W}_i}(S_{{\bar W}_i})}(T_{{\bar W}_i}(-\log S_{{\bar W}_i})),
$$
and if $\# I \geq \varsigma(\bar{Z})$, then $Z\subset \sM$ is a Shimura subvariety of Hodge type for $\SO(d,2)$.
\item[ii)] Assume that the Griffiths-Yukawa coupling vanishes on ${\bar Z}$. If the $W_i$ are Shimura varieties of type  $\SU(d-1,1)$, if
$$
\frac{\deg_ {\omega_{{\bar W}_i}(S_{{\bar W}_i})}(\check{N}_{{\bar W}_i/{\bar Z}})}{\rk \check{N}_{{\bar W}_i/{\bar Z}}}= \frac{\deg_{\omega_{{\bar W}_i}(S_{{\bar W}_i})}(T_{{\bar W}_i}(-\log S_{{\bar W}_i}))}{d+1},
$$
and if $\# I \geq \varsigma(\bar{Z})$, then $Z\subset \sM$ is a Shimura subvariety of Hodge type for $\SU(d,1)$.
\item[iii)]  Assume that ${\bar Z}$ is a surface and that $I=\{1, \ 2\}$.
Assume that 
$$
\sigma_1(\bar{W}_1)\cap \sigma_2(\bar{W}_2)\neq \emptyset
$$ 
and that $\deg \check{N}_{\bar{W}_i/{\bar Z}}=0$.
Then $Z$ is the product of two Shimura curves of Hodge type.
\end{enumerate}
\end{theorem}
Let us start with some preparations for the proof. First of all, 
by Lemma \ref{decomposition} for any coherent sheaf $\sF$ one has
$$
\deg_{E^{2,0}}(\sF)= (d-1)^{d-2} \cdot \deg_{\omega_{\bar W}(S_{\bar W})}(\sF)
\mbox{ \ \ or \ \ }
\mu_{E^{2,0}}(\sF)= d^{d-2} \cdot \mu_{\omega_{\bar W}(S_{\bar W})}(\sF),
$$
depending on the type of $W_i$. In both cases we
are allowed to replace the slope with respect to $\ch_1(\omega_{{\bar W}_i}(S_{{\bar W}_i}))$ by the one with respect to $\ch_1(E^{2,0})$ or $\ch_1(\varphi^*\omega_{\bsM}(S_{\bsM}))$. 

The Mumford compactification $\bsM$ maps to the Baily-Borel compactification $\bsM^*$, and for
$\gamma$ sufficiently large the sheaf $\varphi^*\omega_{\bsM}(S_{\bsM})^\gamma$ is the pullback of a very ample sheaf $\omega_{\bsM^*}$ on $\bsM^*$ (see \cite{m77}). Then the invertible sheaf 
$$
\sL=\det(E^{1,1}_{\bar Z}\otimes E^{2,0}_{\bar Z})^{-\gamma}=\varphi^*(\omega_\bsM(S_\bsM))^\gamma
$$
is semiample. In fact, if $\phi:\hat{Z} \to \bar{Z}$ is a morphism such that
$\bar{Z} \to \bsM^*$ extends to a morphism $\hat{\varphi}:\hat{Z}\to \bsM^*$ the unicity of the Deligne extension implies that
$$
\phi^*\varphi^*\sL=\phi^*\det(E^{1,1}_{\bar Z}\otimes E^{2,0}_{\bar Z})^{-\gamma}=\hat{\varphi}^*\omega_{\bsM^*}^\gamma,
$$  
and hence that $\sL$ is generated by global sections. In the same way, one sees that
$\psi_i^*\sL$ is an invertible sheaf on $\bar{W}_i$ which is generated by global sections.\vspace{.2cm}

The dimension of $\bsM^*\setminus \sM$ is at most one (see \cite{Lo03}, for example). Hence given any component $\Delta_{\bar{Z}}$ of the boundary $S_{\bar{Z}}$ (or of $S_{\bar{W}_i}$) one finds
\begin{gather}\label{boundary}
\Delta.\ch_1(\sL)^2=\ch_1(\sL|_\Delta)^2\equiv 0 \mbox{ \ \ (or \ \ }
\Delta.\ch_1(\psi_i^*\sL)^2\equiv 0\mbox{)}, \mbox{ \ \ and}\\ \notag
\Delta.\ch_1(\sL)=\ch_1(\sL|_\Delta)\equiv 0 \mbox{ \ \ (or \ \ }
\Delta.\ch_1(\psi_i^*\sL)\equiv 0\mbox{)},
\end{gather}
if the dimension of the image of $\Delta_{\bar{Z}}$ or $\Delta_{\bar{W}_i}$ in $\bsM^*$ is a point.\vspace{.1cm}

We will need blowing ups of the  Mumford compactification $\sM$ such that the proper transform $W_i$
meets the boundary transversally outside of codimension two:
\begin{proposition}\label{good}
For some $i\in I$  let $\Delta$ be a component of $S_{\bar{W}_i}$ such that the morphism
$\Delta \to \bsM^*$ is finite (hence $d\leq 3$).
Then there exists a blowing up $\Psi:\bsM_\Delta \to \bsM$, with centers in $S_{\bsM}$, such that:
\begin{enumerate}
\item $S_{\bsM_\Delta}=\bsM_\Delta\setminus \sM$ is a normal crossing divisor and
 $\Omega^1_{\bsM_\Delta}(\log S_{\bsM_\Delta})$ is nef.
\item In a neighborhood of the general point 
of $\Delta$ the rational map $\sigma_i:W_i\to \sM$ extends to an embedding
$\sigma_\Delta:\bar{W}_i\to \bsM_\Delta$ 
whose image intersects $S_{{\bsM}_\Delta}$ transversally.
\end{enumerate}
\end{proposition}
\begin{proof}
Since the first condition holds on $\bsM$ it will hold on $\bsM_\Delta$ if (and only if)
we only blow up strata of the boundary divisors. In fact, this is an easy exercise if the center 
is a point. Since locally along any stratum of $S_{\bsM_\Delta}$ the manifold
$\bsM$ looks like a product one obtains the general case.

For the proof one has to compare the toroidal compactifications, constructed in \cite{AMRT75}, for two local symmetric domains. Fortunately we will only need this in smooth points of boundary components.
A by far more extensive description will be given in Section 2 of the forthcoming article \cite{And07}, and we use this as an excuse, just to sketch the arguments.

Let us fix $i$ and $\Delta$ and drop the lower indices. Recall that $\sM=\Gamma \backslash \SO(n,2) \slash K$ with $\Gamma$ a neat arithmetic group and that $W= \Gamma' \backslash G' \slash K'$ is a local symmetric domain. The inclusion $W\to \sM$ is induced by a homomorphism of groups $G'\to G:=\SO(n,2)$ with a finite kernel. To define the Mumford compactification one needs several data, which we list for $G$. Adding a ${}'$ gives the corresponding notations for $G'$.

First of all, let $D=G\slash K \to D^\vee$ be the embedding of $D$ (we drop the ${}^+$, used in the first section) in its compact dual $D^\vee$. The maximal analytic submanifolds $F$ of $D^\vee \setminus D$ are called the boundary components of $D$. One defines (see \cite[\S 3]{m77}):
\begin{itemize}
\item $N(F):=\{g\in G; \ gF=F\}$.
\item $F$ is rational if $\Gamma\cap N(F)$ is an arithmetic subgroup of $N(F)$.
\end{itemize}
Recall that the boundary of the Baily-Borel compactification $\bsM^*\setminus \sM$ is the disjoint union of finitely many subspaces of the form $ (\Gamma\cap N(F))\backslash F$ for rational boundary components $F$. Next we need
\begin{itemize}
\item $U(F)={}$the center of the unipotent radical $W(F)$ of $N(F)$, as a vector space $\approx \C^k$.
\end{itemize}
The homomorphism $\tau:G'\to G$ extends to a homomorphism $D'^\vee \to D^\vee$ (see \cite{AMRT75} and it induces a map from the set of rational boundary components of $D'$ to the one of $D$. Moreover, the inclusion
$W \subset \sM$ extends to a map $\bar{W}^* \to \bsM^*$ of the Baily-Borel compactifications, compatible with 
the map between boundary components. Fixing some boundary component $F'$ of $D'$ with image $F$,
the characterization of boundary components in \cite[III, \S 3]{AMRT75} shows that $F$ is rational if the same holds for $F'$. Moreover $\tau$ induces compatible morphisms $N(F') \to N(F)$, $W(F')\to W(F)$ and $U(F')\to U(F)$. Furthermore one needs a self-adjoint open convex cone $C(F)\subset U(F)$,
homogeneous under $G$. Again the latter is compatible with $\tau$. The toroidal compactification depends on certain compatible decompositions of the cones $\overline{C(F)}$, for all boundary components. Or, if one uses coordinates, as Mumford does in \cite[\S 3]{m77}, it is given by a certain basis $\{\xi_1,\ldots,\xi_k\}$ of the $\Z$-module $\Gamma \cap U(F)$, with $\xi_1,\ldots \xi_\mu\in C(F)$ and with $\xi_{\mu+1},\ldots,\xi_k\in \overline{C(F)}\setminus C(F)$, for some $\mu\geq 1$.
As we will recall in a moment, each point $q$ in $\bsM$, lying on $S_\bsM$ and with image in
$(\Gamma\cap N(F))\backslash F$, has an analytic neighborhood isomorphic an open subset of $\C^k\times \C^\ell \times F$ with coordinates $(z_1,\ldots,z_k)$ on the first factor. Here the intersection with $\sM$ corresponds to the intersection with ${\C^*}^k\times \C^\ell \times F$, and the different boundary components of $S_\bsM$ map to the zero sets of $z_\iota$ for
some $1\leq \iota  \leq \mu$ (see \cite[page 256, 5)]{m77}). 

The pullback of the cone decomposition defining $\bsM$ gives a cone decomposition for $F'$, hence a second toroidal compactification $\bar{W}'$. In a neighborhood of a general point of $\Delta$ we have a morphism $\bar{W} \to \bar{W}'$, and since both map to the Baily-Borel compactification, this morphism will be an embedding. So we may replace $\bar{W}$ by $\bar{W}'$. As usual we drop the upper index ${}'$, and assume that there is a morphism $\sigma:\bar{W}\to \sM$ of toroidal embeddings.

By assumption, the dimension of $\Delta$ is equal to the one of the rational boundary component $F'$, hence in the description given above one has $k'=1$ and $\ell'=0$, and $U(F')$ is one dimensional.
Hence there is exactly one generator $\xi'$ in the cone $C(F')$. 
Its image in $C(F)$, again denoted by $\xi'$ can be written
as a linear combination
$$
\xi'=\sum_{i=1}^{k} a_i \xi_i \mbox{ \ \ with \ \ } a_i \in \N \mbox{ \ \ and \ \ }
{\rm ggT}\{a_i; \ a_i\neq 0\}=1.
$$
Let us take up the description of local charts, given in \cite[page 256]{m77}:\vspace{.1cm}
$$
\xymatrix{
D'\ar@/^1.5pc/[rr]^{\subset} \ar[r] \ar[d] & (U(F')_\C \times F') \ar[d]^{\gamma'}  \ar@/^1.5pc/[rr]^\alpha & 
D \ar[r]\ar[d] & (U(F)_\C \times \C^{\ell} \times F) \ar[d]^{\gamma} \\ 
\Gamma'\backslash D' \cap U(F')\ar@{^{(}->}[r] \ar[d]& ({\C^*}\times F')\ar@/_1.5pc/[rr]_{\beta \ \ \ \ \ \ \ \ \ }&
\Gamma\backslash D \cap U(F) \ar@{^{(}->}[r] \ar[d]& ({\C^*}^{k}\times \C^{\ell}\times F)\\
W\ar[rr]^\subset & & \sM 
}
$$ 
Here using the basis $\{\xi_i\}$ the vectorspace $U(F)_\C$ is identified with $\C^k$ and
$$
\gamma(x_1,\ldots, x_k)=(e^{2\pi\sqrt{-1}\cdot x_1},\ldots,e^{2\pi\sqrt{-1}\cdot x_k}),
$$
and the same description holds on the left hand side.

The morphisms $\alpha$ and $\beta$ respect the product decomposition, and 
on the first component $\alpha(x)=(a_1\cdot x,\ldots,a_k\cdot x)$. So writing the coordinates 
on ${\C^*}$ and ${\C^*}^k$ as $z$ and $(z_1,\ldots,z_{k})$ one finds
$\beta(z)=(z^{a_1},\ldots,z^{a_k})$, again neglecting the other components.

As above, local neighborhoods of boundary points $q$ of $\bsM$ are given by certain tuples $(F,\{\xi_i\})$,  and the components of the boundary corresponds to the zero set of some of the first $\mu$ components.
We are only interested in those charts, containing the image $p$ of a general point of $\Delta$.
So in the description of $\xi'$ as a linear combination of the $\xi_i$ we can assume that
$$
\ a_1\geq a_2 \geq \cdots \geq a_{\mu'} > a_{\mu'+1}=\cdots=a_k=0,
$$
for some $1\leq \mu' \leq \mu$. Then the image of each branch of $\bar{W}$ in a neighborhood of $p$ is
parameterized by  $\beta(z)=(z^{a_1},\ldots,z^{a_{\mu'}}, 1, \ldots , 1)$, with
${\rm ggT}\{a_1, \cdots, a_{\mu'}\}=1$. 

If $\mu'=1$ we are done. If $\mu'>1$ we blow up the corresponding stratum of $S_\bsM$. After finitely many steps one finds an embedded resolution such that the proper transform meets the new boundary transversally in the smooth locus.
\end{proof}
\begin{corollary}\label{good2}
Let $\Delta$ be an irreducible component of the divisor 
$$
S_{\bar Z}|_{\psi_i({\bar W}_i)} - (S_{\bar Z}|_{\psi_i({\bar W}_i)})_{\rm red}\in {\rm Div}(\psi_i({\bar W}_i))
$$ 
considered in \eqref{tautchern2}.
Then either $\psi_i(\Delta)$ is contained in the intersection $S_\ell \cap S_k$ of two different components $S_\ell$ and $S_k$ of $S_{\bar Z}$ or $\psi_i(\Delta).\ch_1(\sL)^{d-2}=0$ for the invertible sheaf
$\sL$ introduced above. 
\end{corollary}
\begin{proof} We will assume that $\psi_i(\Delta)$ is just contained in one component $S_\ell$, and we will show, that its multiplicity in $S_{\bar Z}|_{\psi_i({\bar W}_i)}$ is at most one. To this aim, we use 
Proposition \ref{good} to choose the blowing up $\bsM_\Delta$ of the given Mumford compactification. By abuse of notations we drop the indices ${}_i$ and ${}_\Delta$. 

In order to verify the Corollary \ref{good2}, we also may replace 
$Z$ by the intersection with $d-2$ general divisors $L_1,\ldots,L_{d-2}$ of the invertible sheaf
$\sL$, introduced above, and correspondingly $\bar{W}$ by the intersection of
$\psi^*L_1,\ldots,\psi^*L_{d-2}$. In fact, if $\Delta$ does not meet this intersection, there is nothing to show. In particular, as remarked in \eqref{boundary} this intersection will be trivial if the fibres of $\Delta \to \bsM^*$ are positive dimensional. As stated in \eqref{boundary} this will always be the case for $d\geq 4$, hence $\dim(\Delta)\geq 2$. Remark that the local transversality of the intersection of $\sigma(\bar{W})$ with $S_{\bsM}$ will be preserved under intersection with general $L_i$.

So let us restrict ourselves to the case where $\bar Z$ is a surface and $\bar W$ a curve.
We have rational maps
$$
\begin{CD}
\bar{W} @> \psi >> \bar{Z} @> \varphi >> \bsM, 
\end{CD}
$$
where $\psi$ and $\sigma=\varphi\circ \psi$ are morphisms and where $\Delta$ is a reduced point of
$\sigma^*S_\bsM$. Choose a minimal blowing up $\Phi:\bar{Z}' \to \bar{Z}$ such that the composite 
$\varphi'=\varphi\circ \Phi$ is a morphism. Of course $\psi$ lifts to a morphism $\psi':\bar{W}\to \bar{Z}'$ near $\Delta$. Since 
$$
\sigma^*S_{\bsM} = \psi'^* \varphi'^*S_{\bsM} \geq \psi'^* S_{\bar{Z}'}, 
$$
the multiplicity of $\Delta$ in $\psi'^* S_{\bar{Z}'}$ is again one and so $\varphi'^*S_{\bsM}$ is reduced 
and non singular in a neighborhood of $\psi'(\Delta)$. Since $\varphi'$ is injective away from the boundary there is a neighborhood of $\psi'(\Delta)$ on which the morphism $\varphi'$
is an embedding whose image meets $S_\bsM$ transversally. So the natural map
$$
\varphi'^*\Omega^1_\bsM(\log S_{\bsM})\longrightarrow \Omega^1_{\bar{Z}'}(\log S_{\bar{Z}'})
$$
will be surjective over this neighborhood. On the other hand, the sheaf on the left hand side is nef,
and the same holds true for its image in $\Omega^1_{\bar{Z}'}(\log S_{\bar{Z}'})$.
So the image has to lie in $\Phi^*\Omega^1_{\bar{Z}}(\log S_{\bar Z})$. Since 
we assumed that $\psi(\Delta)$ is a smooth point of the boundary $S_{\bar{Z}}$, the support of the 
cokernel of
$$
\Phi^*\Omega^1_{\bar{Z}}(\log S_{\bar Z}) \subset
\Omega^1_{\bar{Z}'}(\log S_{\bar{Z}'})
$$
contains the whole exceptional locus. So there is no blowing up, $\varphi$ is a morphism, and
$\Delta$ is reduced in $\psi^* S_{\bar{Z}}$.
\end{proof}
\begin{proof}[Proof of Theorem \ref{3.2}] 
We start with parts i) and ii). 

Let us choose morphisms $\phi_i:\hat{W}_i \to \bar{W}_i$ such that the rational map
$\psi_i$ lifts to a morphism $\hat{\psi}_i:\hat{W}_i \to \bar{Z}$.
Since $\check{N}_{{\bar W}_i/{\bar Z}}\hookrightarrow{\phi_i}_*\phi_i^*\check{N}_{{\bar W}_i/{\bar Z}}
\hookrightarrow \check{N}_{{\bar W}_i/{\bar Z}}^\natural$ one obtains
$$
\deg_{\hat{\psi}_i^*\sL}(\phi_i^* \check{N}_{{\bar W}_i/{\bar Z}}/_{\rm torsion})=
\deg_{\psi_i^*\sL}({\phi_i}_*\phi_i^*\check{N}_{{\bar W}_i/{\bar Z}})=\deg(\check{N}_{\bar{W}_i/\bar{Z}})..
$$
As $\varphi|_Z$ is an embedding, the Assumptions made in \ref{assSect3} hold and Theorem \ref{1.6}, i) and ii) and the projection formula imply that for all $i\in I$ 
$$
\deg_{\hat{\psi}_i^*\sL}(\phi_i^* \check{N}_{{\bar W}_i/{\bar Z}}/_{\rm torsion})<0.
$$
\begin{claim}\label{large}
If $\dim(Z)=d\geq 4$ then $\deg_{\hat{\psi}_i^*\sL}(\phi_i^* \check{N}_{{\bar W}_i/{\bar Z}}/_{\rm torsion})
=(\hat{\psi}_i({\hat W}_i))^2.\ch_1(\sL)^{d-2}$. In particular $(\hat{\psi}_i({\hat W}_i))^2.\ch_1(\sL)^{d-2}<0$.
\end{claim}
\begin{proof}
By \eqref{boundary} for any component $\Delta_j$ of $S_{{\bar W}_i}$ one has $\Delta_j.\ch_1(\sL)^{d-2}=0$, and the Claim \ref{large} follows from
\eqref{tautchern2}. 
\end{proof}
Unfortunately, for $d\leq 3$ the corresponding equality is only guaranteed under the additional assumption made in the Notations \ref{3.2n}, (2).
\begin{claim}\label{cases}
Assume that $\# I \geq \rho^2 + \rho + 1$. Then one of the following conditions hold:
\begin{enumerate}
\item[a.] There exists a linear combination 
$$
D=\sum_{i\in I} a_i\cdot \hat{\psi}_i({\hat W}_i),
$$
with $a_i\geq 0$, such that $D^2.\ch_1(\sL)^{d-2}>0$.
\item[b.] For all $i\in I$ the intersection number $(\hat{\psi}_i({\hat W}_i))^2.\ch_1(\sL)^{d-2}\leq 0$, and
for some $\iota\in I$ and all $i\in I$ the intersection numbers $(\hat{\psi}_\iota({\hat W}_\iota)).(\hat{\psi}_i({\hat W}_i)) .\ch_1(\sL)^{d-2}= 0$. 
\end{enumerate}
\end{claim}
\begin{proof}
If $(\hat{\psi}_i({\hat W}_i))^2.\ch_1(\sL)^{d-2}>0$ for some $i\in I$ the condition a) obviously holds true. 
If there are two indices $i$ and $j$ with 
$$
(\hat{\psi}_\iota({\hat W}_\iota)).(\hat{\psi}_i({\hat W}_i)) .\ch_1(\sL)^{d-2}> 0\mbox{ \ \  and \ \ } (\hat{\psi}_\iota ({\hat W}_\iota ))^2.\ch_1(\sL)^{d-2}= 0
$$ 
one can choose $D=\alpha\cdot \hat{\psi}_\iota({\hat W}_\iota) + \hat{\psi}_i({\hat W}_i)$ for $\alpha \gg1$, and again a) holds. 

If $(\hat{\psi}_i({\hat W}_i))^2.\ch_1(\sL)^{d-2} < 0$ for all $i \in I$, one can use Theorem
\ref{2.1} to verify a).

All the remaining cases are covered by b).
\end{proof}
\begin{claim}\label{modification}
Assume that we are in case b) in Claim \ref{cases}. Then there exists a blowing up
$\Phi:\bar{Z}' \to \bar{Z}$, which satisfies again the assumptions made in \ref{assSect4}, and 
for which one is in case a).
\end{claim}
\begin{proof}
We choose $\Phi:\bar{Z}' \to \bar{Z}$ to be the successive blowing up of the non-empty intersections
$S_\ell\cap S_k$ of different components $S_\ell$ and $S_k$ of $S_{\bar Z}$.

By Claim \ref{large} the condition (1) in the Notations \ref{3.2n} excludes the case b) in Claim 
\ref{cases}. For the other two conditions (2), and (3) stated there, the same is obvious.
So the definition of $\varsigma(\bar{Z})$ and the assumptions made in Theorem \ref{3.2}, i) and ii)
say that 
$$
\# I \geq (\rho+\delta)^2 + \rho + \delta + 1=\rho(\bar{Z}')^2 + \rho(\bar{Z})' + 1.
$$
Of course we may assume that $\phi_i:\hat{W}_i \to \bar{W}_i$ is chosen such that $\hat{\psi}_i:\hat{W}_i\to \bar{Z}$ factors through $\hat{\psi}'_i:\hat{W}_i\to \bar{Z}'$.

As we had seen already in the proof of Proposition \ref{good} the divisor $S_{{\bar Z}'}={\bar Z}'\setminus Z=\Phi^{-1}(S_{\bar Z})$ is still a normal crossing divisor and
\begin{equation*}
\Omega^1_{{\bar Z}'}(\log S_{{\bar Z}'})=\Phi^*\Omega^1_{{\bar Z}}(\log S_{{\bar Z}}).
\end{equation*}
So $\bar{Z}$ again satisfies the assumptions made in \ref{assSect4}.

Let $X \subset {\bar Z}$ be the smooth surface obtained by intersecting $d-2$ zero-divisors of general sections of $\sL$ and let $X'$ be its preimage in $\bar{Z}'$. Then for all $i\in I$
$$
(\hat{\psi}_i({\hat W}_i))^2.\ch_1(\sL)^{d-2}=(\hat{\psi}_i({\hat W}_i)|_X)^2  \geq 
(\hat{\psi}'_i({\hat W}_i)|_{X'})^2.
$$
If this is an equality, none of the points lying on $\hat{\psi}_i({\hat W}_i)$ is blown up.

On the other hand, if $\bar{W}_i$ is one of the divisors with 
$(\hat{\psi}_i({\hat W}_i))^2.\ch_1(\sL)^{d-2} = 0$, hence with $(\hat{\psi}_i({\hat W}_i)|_X)^2=0$, 
then for some component $\Delta$ of $\hat{\psi}_i^* S_{\bar{Z}} - S_{{\hat W}_i}$ one has
$\hat{\psi}_i(\Delta).\ch_1(\sL)^{d-2} > 0$. By Corollary \ref{good2} $\hat{\psi}_i(\Delta)$ contains at least one of the intersections $S_{\ell,k}$, hence its restriction to $X$ is blown up.

Then for all $i\in I$ one finds $(\hat{\psi}'_i({\hat W}_i))^2.\ch_1(\Phi^*\sL)^{d-2}<0$ and the last condition in Claim \ref{cases}, b) is violated for $i=\iota$. 
\end{proof}
From now on, we will replace $\bar{Z}'$ by $\bar{Z}$ and assume by abuse of notations that
we are in case a) in Claim \ref{cases}, hence for some effective linear combination $D$ of the
$\psi_i(\bar{W}_i)$ we have $D^2.\ch_1(\sL)^{d-2}>0$.

Consider the saturated Higgs subsheaf $(F,\theta) \subset (E_{\bar Z},\theta_{\bar Z})$, which is generated by $E^{2,0}_{\bar Z}$, as described in \eqref{FHiggs1} and \eqref{FHiggs2}. 
Since $(F,\theta) \subset (E_{{\bar Z}},\theta_{{\bar Z}})$ is a Higgs subbundle of a Higgs bundle arising from a variation of Hodge structures with logarithmic singularity along $S_{{\bar Z}}$, the sheaf $\det(F)$ is negative semi-definite in the sense that the curvature of the Hodge metric is negative semi-definite. 

Then by the projection formula
$$
\deg_{\hat{\psi}_i^*\sL}(\phi_i^*\psi_i^*F/_{\rm torsion})=
\deg_{\psi_i^*\sL}({\phi_i}_*\phi_i^*(\psi_i^*F/_{\rm torsion})).
$$
Since $\psi_i^*F\hookrightarrow{\phi_i}_*\phi_i^*(\psi_i^*F/_{\rm torsion}) \hookrightarrow
(\psi_i^*F)^\natural$, it follows from Addendum \ref{3.1} that this degree is zero. The projection formula
implies that 
$$
\ch_1(F).\hat{\psi}_i(\hat{W}_i).\ch_1(\sL)^{d-2}=0 \mbox{ \ \ hence \ \ }
\ch_1(F).D.\ch_1(\sL)^{d-2}=0. 
$$
Let again $X \subset {\bar Z}$ be the smooth surface obtained by intersecting $d-2$ zero-divisors of general sections of $\sL$. Then  $\ch_1(F|_X) . D\cap X =0$. Since 
$$
(D\cap X)^2 > 0 \mbox{ \ \ and \ \ } (\det F|_X)^2 \geq 0,
$$
the Hodge index theorem implies that $\ch_1(\det F|_X)=0$.

Since $\ch_1(F)$ is represented by a negative semi-definite Chern form the latter implies that
$\ch_1(F)=0$. By Simpson's poly-stability the Higgs subbundle $(F,\theta)\subset (E,\theta)$ is a direct factor. Its complement has a trivial Higgs bundle, hence it is induced by a unitary local subsystem.
By Lemma \ref{decomposition} we obtain i) and ii).\\[.2cm]
In Case iii) $\psi_i:\bar{W}_i \to  \bar{Z}$ are morphisms. By assumption
$$
(\psi_i \bar{W}_i)^2 = \deg(\check{N}_{\bar{W}_i/{\bar Z}})=0 \mbox{ \ \ and \ \ }
\psi_1(\bar{W}_1) . \psi_2(\bar{W}_2)>0.
$$
Part iii) of the Addendum \ref{3.1} implies that 
$$
\ch_1( F). (\psi_1(\bar{W}_1)+\psi_2(\bar{W}_2))=\ch_1(F|_{\psi_1(\bar{W}_1)})+\ch_1(F|_{\psi_2(\bar{W}_2)})=0.
$$
Since $\ch_1( F)$ is negative semi-definite and since $(\psi_1(\bar{W}_1)+\psi_2(\bar{W}_2))^2>0$ the Hodge index theorem tells us that $\ch_1(F)=0$. As before this implies that $(F,\theta)$ is a direct factor of 
$(E_{\bar Z},\theta_{\bar Z})$ and $Z\subset \sM$ is a Shimura surface of Hodge type. 

By Theorem \ref{1.6}, iii), the Griffiths-Yukawa coupling on ${\bar Z}$ does not vanish.
Thus, $Z$ is a generalized Hilbert modular surface, necessarily rigid. 
$Z$ can not be a genuine Hilbert modular surface, since $Z$ contains Shimura curves with vanishing Griffiths-Yukawa coupling. \end{proof}

\begin{proof}[Proof of Theorems \ref{HHP0R} and \ref{CharShimura}]
As we have seen in the proof of Theorem \ref{HHP0} at the end of Section \ref{HHPR}, 
on a surface $\bar{Y}=\bar{Z}$ the equality of Chern numbers
\eqref{tautchern} implies that the inequalities i), ii) and iii) in Theorem \ref{1.6}
coincide with the inequalities \eqref{eqp4}, \eqref{eqp5} and \eqref{eqp6} in Theorem
\ref{HHP0R}. For the same reason, Theorem  \ref{CharShimura} is just a special case of
Theorem \ref{3.2}.
\end{proof}

\end{document}